\documentclass[12pt]{amsart}
\usepackage[lmargin=1.1in,rmargin=1.1in,tmargin=1.5in,bmargin=1.3in]{geometry}
\usepackage{amsmath,amssymb,amsfonts}
\usepackage{amsthm}
\usepackage{graphicx}
\usepackage{hyperref}

\theoremstyle{plain}
\newtheorem{theorem}{Theorem}[section]
\newtheorem{lemma}[theorem]{Lemma}
\newtheorem{proposition}[theorem]{Proposition}
\newtheorem{corollary}[theorem]{Corollary}

\newtheorem{problem}[theorem]{Problem}

\theoremstyle{definition}
\newtheorem{definition}[theorem]{Definition}
\newtheorem{remark}[theorem]{Remark}
\newtheorem{convention}[theorem]{Conventions}

\title[A cork of the rational surface the second Betti number 9]{A cork of the rational surface with the second Betti number 9}

\author{Yohei Wakamaki}
\address{Department of Pure and Applied Mathematics, Graduate School of Information Science and Technology, Osaka University, 1-5 Yamadaoka, Suita, Osaka 565-0871, Japan}
\email{y-wakamaki@ist.osaka-u.ac.jp}

\date{July 31, 2023. \textit{Revised}: September 26, 2024.}

\begin{document}

\begin{abstract}
  We provide the first explicit example of a cork of $\mathbf{CP}^2 \# 8\overline{\mathbf{CP}^2}$. This result gives the current smallest second Betti number of a standard simply-connected closed $4$-manifold for which an explicit cork has been found.
\end{abstract}
\maketitle

\section{Introduction}\label{intro}
Throughout this article, we assume that all manifolds are smooth and oriented, and all maps are smooth unless otherwise stated. 

One of the fascinating problems in $4$-dimensional topology asks whether a simply-connected closed $4$-manifold with a small second Betti number $b_2$ admits an exotic smooth structure. The interest in this problem stems from the fact that constructing exotic smooth structures on such $4$-manifolds is much more challenging than on those with large $b_2$. The first example of an exotic smooth structure on simply-connected closed $4$-manifolds was discovered by Donaldson \cite{1985Donalson, 1987Donaldson}. A few years later, Friedman and Morgan \cite{1988Friedman-Morgan} proved that for each integer $m \geq 10$, there exists a simply-connected closed $4$-manifold with $b_2=m$ such that it admits infinitely many exotic smooth structures. After their result, many experts \cite{1989Kotschick,2005Park,2005Stipsicz-Szabo,2005Park-Stipsicz-Szabo,2008Akhmedov-Park,2010Akhmedov-Park} have contributed to lowering the known minimal value of $b_2$ of a simply-connected closed $4$-manifold on which an exotic smooth structure exists. Currently, we know that for each integer $m\geq 3$, there exists a simply-connected closed $4$-manifold with $b_2= m$ that admits infinitely many exotic smooth structures.

Regarding the exotic smooth structures on simply-connected closed $4$-manifolds, the study of smooth structures of $4$-manifolds by using the concept of \textit{cork} (see Definition \ref{def:cork}) plays an important role. Due to the work of Curtis--Freedman--Hsiang--Stong \cite{1996Curtis-Freedman-Hsiang-Stong} and independently Matveyev \cite{1996Matveyev}, for any exotic pair $(X,Y)$ of simply-connected closed $4$-manifolds, there exist a cork $(C,\tau)$ and an embedding $i$ of $C$ into $X$ such that $Y$ is diffeomorphic to the \textit{cork twist of $X$ along $(C,\tau, i)$}, i.e., the $4$-manifold obtained by cutting out the embedded copy $i(C)$ in $X$ and regluing $C$ via the map $i\circ \tau$. In other words, one can obtain any exotic smooth structure of $X$ by a cork twist of $X$. When the cork twist of $X$ along $(C, \tau, i)$ is exotic to $X$, we say $i$ is an \textit{effective embedding of $(C,\tau)$ into $X$}. A cork $(C,\tau)$ is called a \textit{cork of} $X$ if there exists an effective embedding of $(C, \tau)$ into $X$.

Despite the importance of corks of simply-connected closed $4$-manifolds, we have relatively few explicit examples of them. In particular, the minimal value of $b_2$ of \textit{standard} simply-connected closed $4$-manifolds for which an explicit cork had been found to date was $10$ \cite{2012Akbulut}. A simply-connected closed $4$-manifolds is called standard if it is obtained as the connected sum of finitely many copies of $\mathbf{CP}^2, \overline{\mathbf{CP}^2}, S^2\times S^2, K3$ and $\overline{K3}$. As is well-known, if the celebrated $11/8$-conjecture \cite{1982Matsumoto} is true, then it follows that any simply-connected closed $4$-manifold is homeomorphic to one of the standard simply-connected closed $4$-manifolds.
  
  This situation naturally leads us to the following problem, which can be regarded as a cork version of the problem we raised at the beginning of this paper.

  \begin{problem}\label{problem}
   Find an effective embedding of a cork into a standard simply-connected closed $4$-manifold with $b_2\leq 9$.
  \end{problem} 
  
   We remark that an effective embedding of the cork into a non-standard simply-connected closed $4$-manifold with $b_2=9$ has already been found by Akbulut and Yasui \cite[Remark 6.2]{2008Akbulut-Yasui}, but it is unknown whether its cork twist results in a standard simply-connected closed $4$-manifold. By the definition of cork twist, if the cork twist along their cork results in a standard simply-connected closed $4$-manifold, it immediately follows that their cork has an effective embedding into a standard simply-connected closed $4$-manifold.

  In this article, we prove the following theorem by finding an effective embedding of a cork in Figure \ref{fig:W2} into a standard $4$-manifold. As a result, we provide an answer to Problem \ref{problem} in the case $b_2=9$.
  
  \begin{theorem}\label{thm:cork}
    There exists an effective embedding of the cork $(W_2,f_2)$ into $\mathbf{CP}^2 \# 8\overline{\mathbf{CP}^2}$, i.e., the cork $(W_2,f_2)$ is a cork of $\mathbf{CP}^2 \# 8\overline{\mathbf{CP}^2}$.
  \end{theorem}

  \begin{figure}[htbp]
    \includegraphics*[width=6cm]{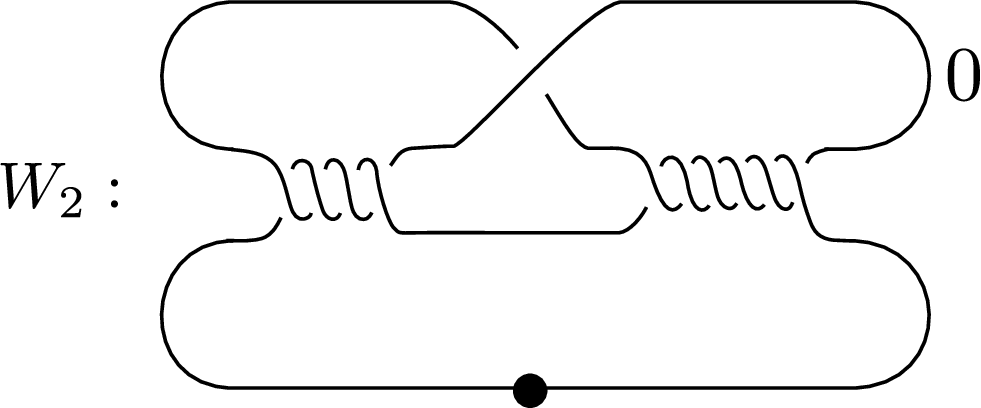}
    \caption{The cork $(W_2,f_2)$. The involution $f_2$ of $\partial W_2$ is defined by exchanging the zero and the dot in this diagram.}
    \label{fig:W2}
  \end{figure}

  The proof of this theorem is divided into two parts. The first is giving an explicit Kirby diagram (Figure \ref{fig:R8_1}) of an exotic $\mathbf{CP}^2 \# 8\overline{\mathbf{CP}^2}$ obtained by Yasui's construction \cite[Corollary 5.2]{2010Yasui} which uses the rational blowdown technique \cite{1997Fintushel-Stern}. We denote this manifold as $R_8$. Note that Yasui \cite{2010Yasui} explicitly described a procedure to give a Kirby diagram of a $4$-manifold obtained by his construction. However, no explicit Kirby diagram of an exotic $\mathbf{CP}^2 \# 8\overline{\mathbf{CP}^2}$ obtained by Yasui's construction was given before. We follow Yasui's procedure to give a diagram of $R_8$ with some modifications. The second is finding the diagram which contains an embedded copy of $W_2$ such that the cork twists along $(W_2,f_2)$ results in $\mathbf{CP}^2 \# 8\overline{\mathbf{CP}^2}$. Such a diagram is described in Figure \ref{fig:R8_2}. 

  The following follows from Theorem \ref{thm:cork}.

  \begin{corollary}\label{cor:stab}
    The $4$-manifold $R_8$ becomes diffeomorphic to $2\mathbf{CP}^2 \# 9\overline{\mathbf{CP}^2}$ after taking a connected sum with $S^2 \times S^2$.
  \end{corollary}

  This corollary is related to the famous open problem asking whether every exotic pair of simply-connected closed $4$-manifolds becomes diffeomorphic after one stabilization, i.e., taking a connected sum with $S^2 \times S^2$. It is well-known that, due to the theorem of Wall \cite{1964Wall}, every exotic pair of simply-connected closed $4$-manifolds becomes diffeomorphic after sufficiently many stabilizations, and it has been proved that only one stabilization is enough in many cases. To the best of the author's knowledge, among the exotic pairs of simply-connected closed $4$-manifolds whose Kirby diagrams are explicitly given, the pair $(\mathbf{CP}^2 \allowbreak\# \allowbreak8\overline{\mathbf{CP}^2}, R_8)$ is currently the smallest example in terms of $b_2$ that becomes diffeomorphic after one stabilization. It is not clear to the author whether other examples of exotic pairs of simply-connected closed $4$-manifolds with $b_2\leq 8$ known to date become diffeomorphic after one stabilization. We note that many variants of the problem about one stabilization have recently been answered negatively. For details, see \cite{2020Lin,2021Lin-Mukherjee,2022Guth,2022Konno-Mukherjee-Taniguchi,2022Kang,2023Hayden,2023Hayden-Kang-Mukherjee,2023Konno-Mallick-Taniguchi1,2023Konno-Mallick-Taniguchi2,2023Auckly}.

  \begin{remark}\label{rmk:stab}
    We can also prove an exotic $\mathbf{CP}^2 \# k\overline{\mathbf{CP}^2} (k=5,6,7,9)$ obtained by Yasui's construction becomes diffeomorphic to $2\mathbf{CP}^2 \# (k+1)\overline{\mathbf{CP}^2}$ after one stabilization. Furthermore, after posting the first version of this article on arXiv, Rafael Torres informed the author that a stronger version of Corollary \ref{cor:stab} holds. Namely, the manifold $R_8$ becomes diffeomorphic to $2\mathbf{CP}^2 \# 8\overline{\mathbf{CP}^2}$ after taking a connected sum with $\mathbf{CP}^2$. It is possible to apply his idea to an exotic $\mathbf{CP}^2 \# k\overline{\mathbf{CP}^2} (k=6,7,9)$ obtained by Yasui's construction.  These results will be discussed in the forthcoming paper \cite{Wakamaki}.
  \end{remark}

  \subsection{Acknowledgement.}
    The author wishes to express his deepest gratitude to his advisor, Kouichi Yasui, for his patience, encouragement, and numerous valuable suggestions, including the topic of this paper. He is grateful to Rafael Torres for generously informing the author of a stronger version of Corollary \ref{cor:stab} mentioned in Remark \ref{rmk:stab}. He is thankful to the anonymous referees for their careful reading, many suggestions, and pointing out mistakes in the original manuscript. He also thanks Natsuya Takahashi for many valuable conversations and comments on the draft of this paper, and Yuichi Yamada for his interest in this study.
  
\section{Rational blowdown and Yasui's small exotic rational surfaces}\label{section:construction}
  We start this section by reviewing the definition of the rational blowdown of $4$-manifold, which was introduced by Fintushel and Stern \cite{1997Fintushel-Stern}. Then, we recall Yasui's construction \cite{2010Yasui} of an exotic $\mathbf{CP}^2 \# 8 \overline{\mathbf{CP}^2}$ that uses the rational blowdown. To find a cork embedded in an exotic $\mathbf{CP}^2 \# 8 \overline{\mathbf{CP}^2}$ in Section \ref{section:finding_a_cork}, we follow Yasui's construction \cite{2010Yasui} with small modifications (see Remark \ref{plumbing}) and give an explicit diagram of the $4$-manifold.
  
  \begin{definition}
    For each integer $p\ge 2$, let $C_p$ and $B_p$ be the compact $4$-manifolds with boundary defined by the Kirby diagrams in Figure \ref{fig:Cp&Bp}. Here $u_i$ in Figure \ref{fig:Cp&Bp} represents the elements of $H_2(C_p;\mathbf{Z})$ given by the corresponding $2$-handles. (i.e., $u_{p-1}^2 = -p-2, u_i^2= -2, \textrm{and}~ u_i \cdot u_{i+1} = +1 ~(1\leq i \leq p-2)$.) 
    
    \begin{figure}[htbp]
      \begin{center}
        \includegraphics[width=11cm]{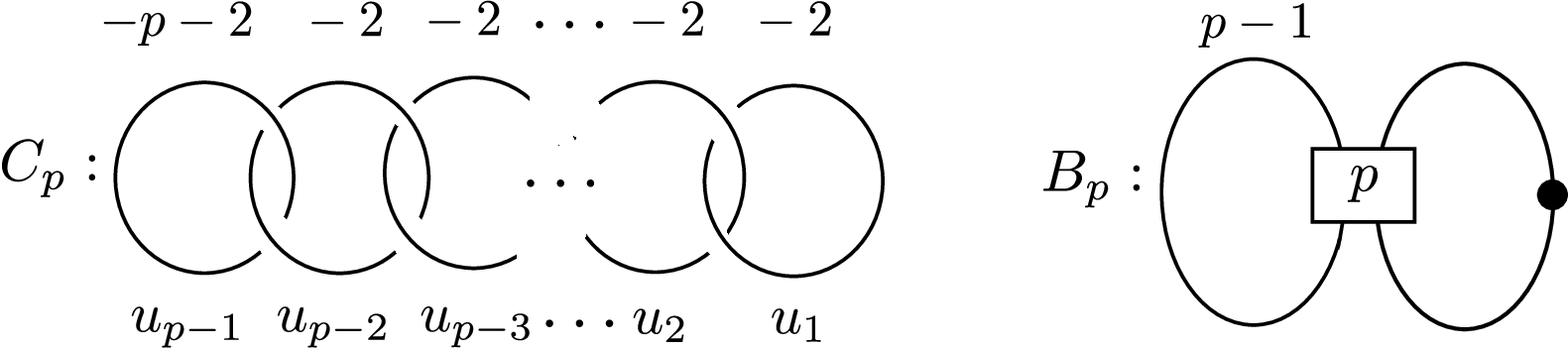}
        \caption{$C_p$ and $B_p$.}
        \label{fig:Cp&Bp}
      \end{center}
    \end{figure}
    
  \end{definition}
  
  \begin{definition}[\cite{1997Fintushel-Stern},\cite{1999Gompf-Stipsicz}]\label{def:RBD}
    Let $X$ be a compact $4$-manifold and $C$ be an embedded copy of $C_p$ in $X$. The $4$-manifold $X_{(p)}=X-(\textrm{int} C) \cup_\partial B_p$ is called the \textit{rational blowdown} of $X$ along $C$. This operation is well-defined since any self-diffeomorphism of $\partial B_p$ extends over $B_p$ (For details, see \cite[Section 8.5]{1999Gompf-Stipsicz}).
  \end{definition}
  
  Yasui's construction \cite{2010Yasui} begins with the following proposition. Although he constructed exotic $\mathbf{CP}^2 \# k\overline{\mathbf{CP}^2}$ for $5\leq k \leq 9$, we only focus on the case $k=8$. 
  
  \begin{proposition}[{\cite[Proposition 3.1 (1)]{2010Yasui}}]
    For $a \geq 1$, the complex projective plane $\mathbf{CP}^2$ admits the handle decomposition in Figure \ref{fig:CP2}.
    \begin{figure}[htbp]
      \begin{center}
        \includegraphics[width=10cm]{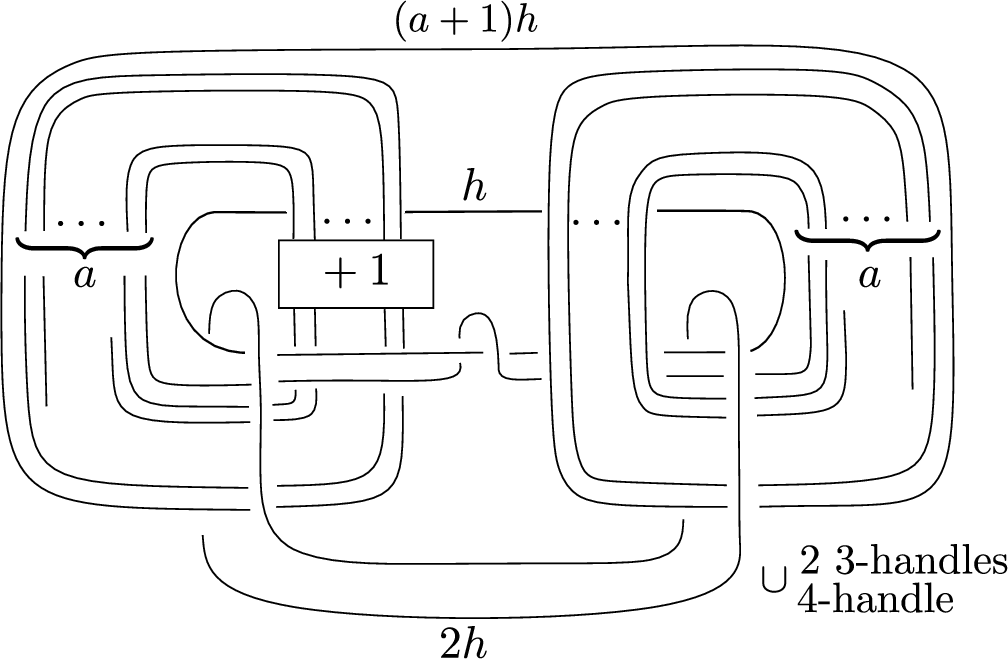}
        \caption{$\mathbf{CP}^2\quad (a\ge 1)$.}
        \label{fig:CP2}
      \end{center}
    \end{figure}
  \end{proposition}
  
  \begin{convention}\label{convention}
    (1) In the figures below, we often draw only the local pictures of Kirby diagrams. We assume that the parts not drawn in the diagrams are naturally inherited from the previous diagrams and always fixed.\\
    (2) In order to indicate the bands for the handle slides, we sometimes draw arrows in the diagrams as on the left in Figure \ref{fig:handle_slide}. Figure \ref{fig:handle_slide} only shows the cases when the attaching circles of $2$-handles are unknots and unlinked. As in Figure \ref{fig:handle_slide}, the shape of the arrow determines the band for the handle slide. The arrow with a positive twist on the left in Figure \ref{fig:handle_slide} (b) will only appear in Figure \ref{fig:R8_corktwist_local_1} and Figure \ref{fig:R8_corktwist_local_3}. As usual, the boxes with integer $m$ in figures stand for $m$ right-handed full twists if $m$ is positive and $|m|$ left-handed full twists if $m$ is negative.
    \begin{figure}[htbp]
      \begin{center}
        \includegraphics[width=10cm]{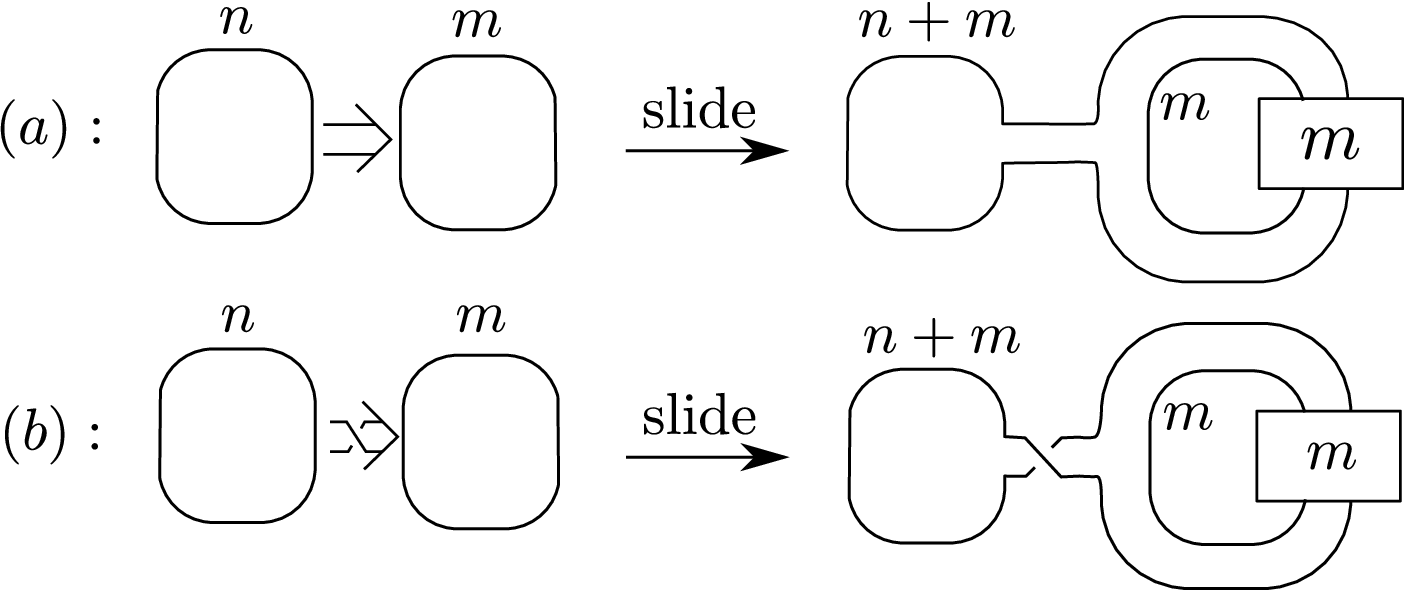}
        \caption{The arrows (on the left) determine the bands (on the right) for handle slides.}
        \label{fig:handle_slide}
      \end{center}
    \end{figure}
    \\
    (3) We often represent the framings of the $2$-handles by the second homology classes corresponding to the $2$-handles. It is because we need the information of homology classes of certain $2$-handles to prove Theorem \ref{thm:R8}. When the framings of the $2$-handles are represented by the second homology classes, we mention the $2$-handles or the attaching circles of $2$-handles by their homology classes. Note that one can obtain the usual framing coefficients of the $2$-handles by squaring their homology classes.\\
    (4) We denote the natural orthogonal basis of $H_2(\mathbf{CP}^2 \# k \overline{\mathbf{CP}^2} ; \mathbf{Z}) = \allowbreak H_2(\mathbf{CP}^2 ; \allowbreak \mathbf{Z}) \allowbreak \oplus_k \allowbreak H_2(\overline{\mathbf{CP}^2} \allowbreak ;\allowbreak \mathbf{Z})$ by $h, e_1, e_2, \dots, e_k$~(i.e., $h^2 = 1, e_i^2 = -1, h \cdot e_i = 0, ~\textrm{and}~ e_i \cdot e_j = 0 ( 1\leq i \neq j \leq k)$).
  \end{convention}

  \begin{proposition}[cf. {\cite[Proposition 3.2 (1), $a=4$]{2010Yasui}}]\label{prop:CP2-14CP2}
    $\mathbf{CP}^2 \# 14\overline{\mathbf{CP}^2}$ admits the handle decomposition in Figure \ref{fig:CP2-14CP2_1}.

    \begin{figure}[htbp]
      \begin{center}
        \includegraphics[width=14cm]{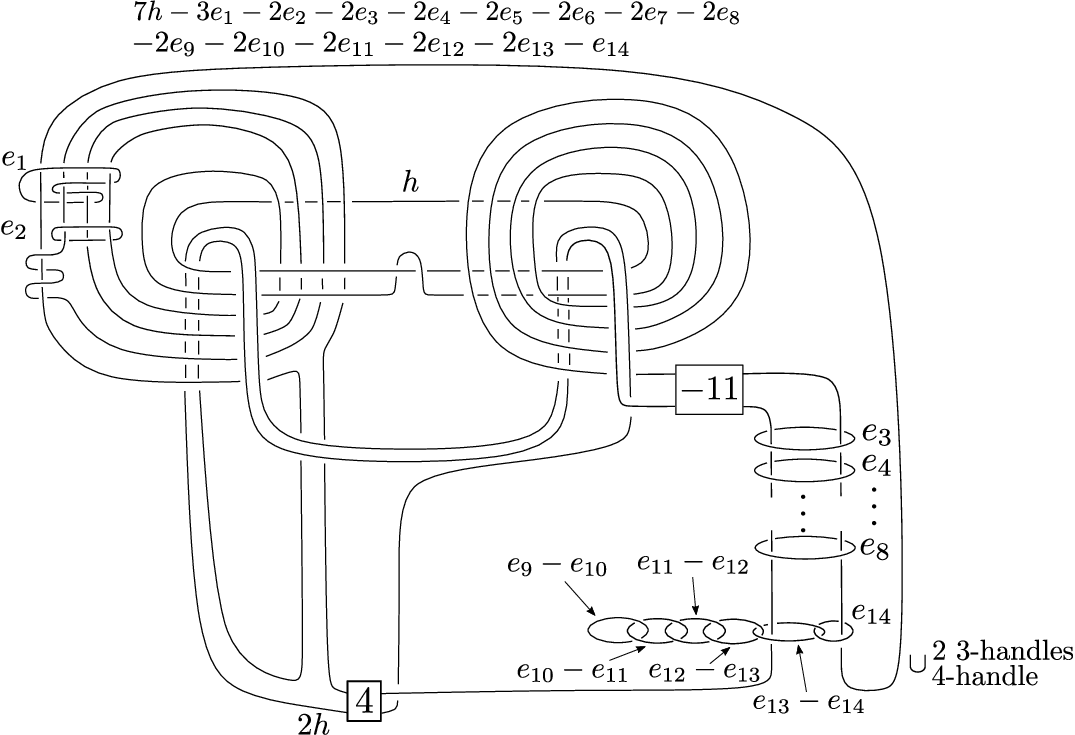}
        \caption{A diagram of $\mathbf{CP}^2 \# 14\overline{\mathbf{CP}^2}$.}
        \label{fig:CP2-14CP2_1}
      \end{center}
    \end{figure}
  \end{proposition}
  
  \begin{proof}
    The top left picture of Figure \ref{fig:deformation_to_obtain_CP2-2CP2} shows the neighborhood of the full twist in Figure \ref{fig:CP2} for the case $a=4$. By the procedure described in Figure \ref{fig:deformation_to_obtain_CP2-2CP2}, we obtain a Kirby diagram of $\mathbf{CP}^2 \# 2\overline{\mathbf{CP}^2}$ shown in Figure \ref{fig:CP2-2CP2_1}. We isotope this diagram to obtain Figure \ref{fig:CP2-2CP2_2}. Then we can move one of the kinks on the left side of this diagram to obtain Figure \ref{fig:CP2-2CP2_3}. By another isotopy, we obtain Figure \ref{fig:CP2-2CP2_4} and slide $5h-3e_1-2e_2$ over $2h$ to obtain Figure \ref{fig:CP2-2CP2_5}. Now we isotope this diagram and perform $11$ blowups to obtain in Figure \ref{fig:CP2-13CP2_1}. We slide $e_9$ over $e_{10}$, $e_{10}$ over $e_{11}$, $e_{11}$ over $e_{12}$, and $e_{12}$ over $e_{13}$. Then we obtain Figure \ref{fig:CP2-13CP2_2}. If we perform a blowup here, we obtain Figure \ref{fig:CP2-14CP2_1}, and this diagram represents a handle decomposition of $\mathbf{CP}^2 \# 14\overline{\mathbf{CP}^2}$.
  \end{proof}
    
    \begin{figure}[htbp]
      \begin{center}
        \includegraphics[width=10cm]{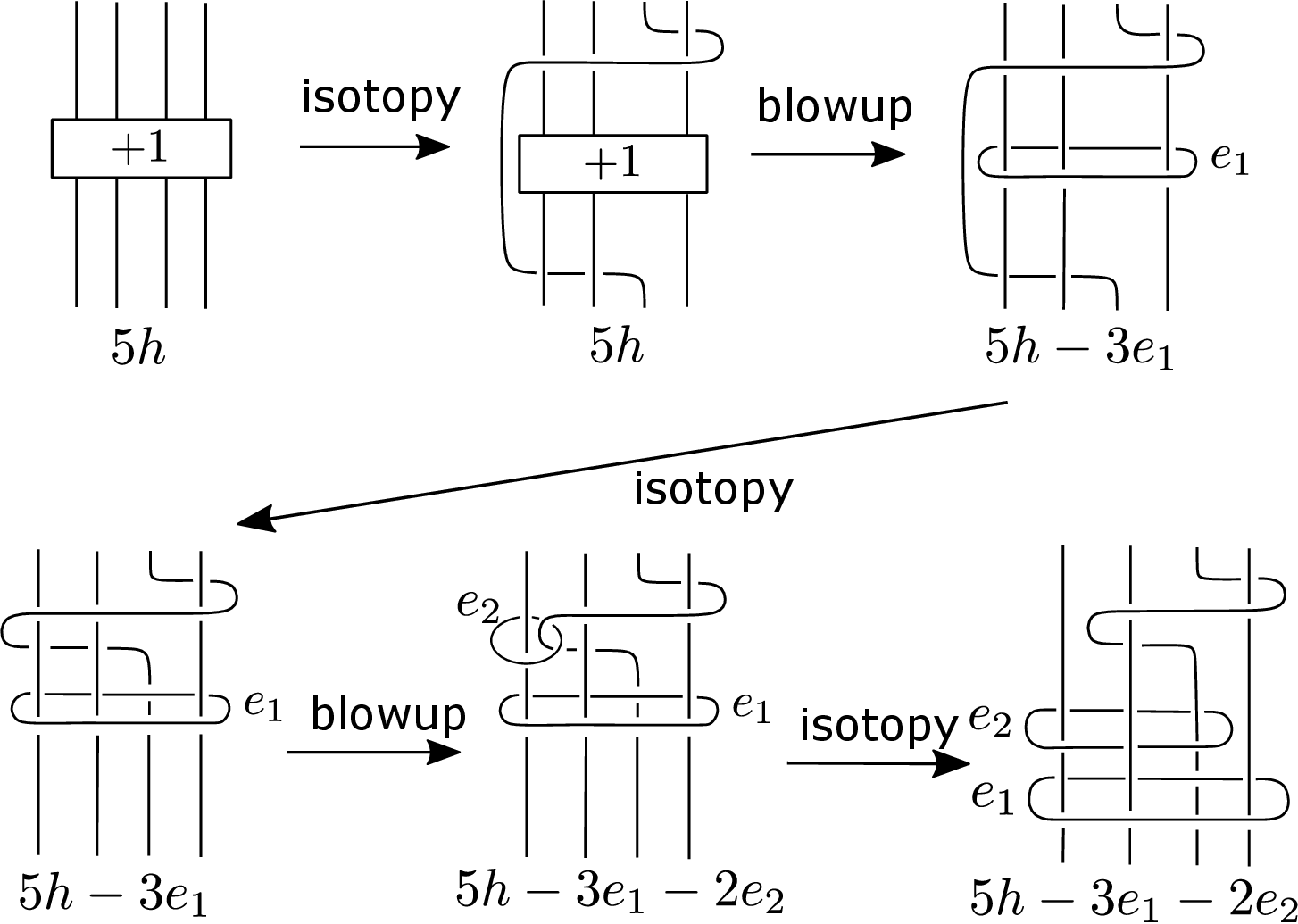}
        \caption{Isotopies and blowups applied to the diagram in a neighborhood of the full twist in Figure \ref{fig:CP2}.}
        \label{fig:deformation_to_obtain_CP2-2CP2}
      \end{center}
    \end{figure}
    
    \begin{figure}[htbp]
      \begin{center}
        \includegraphics[width=12cm]{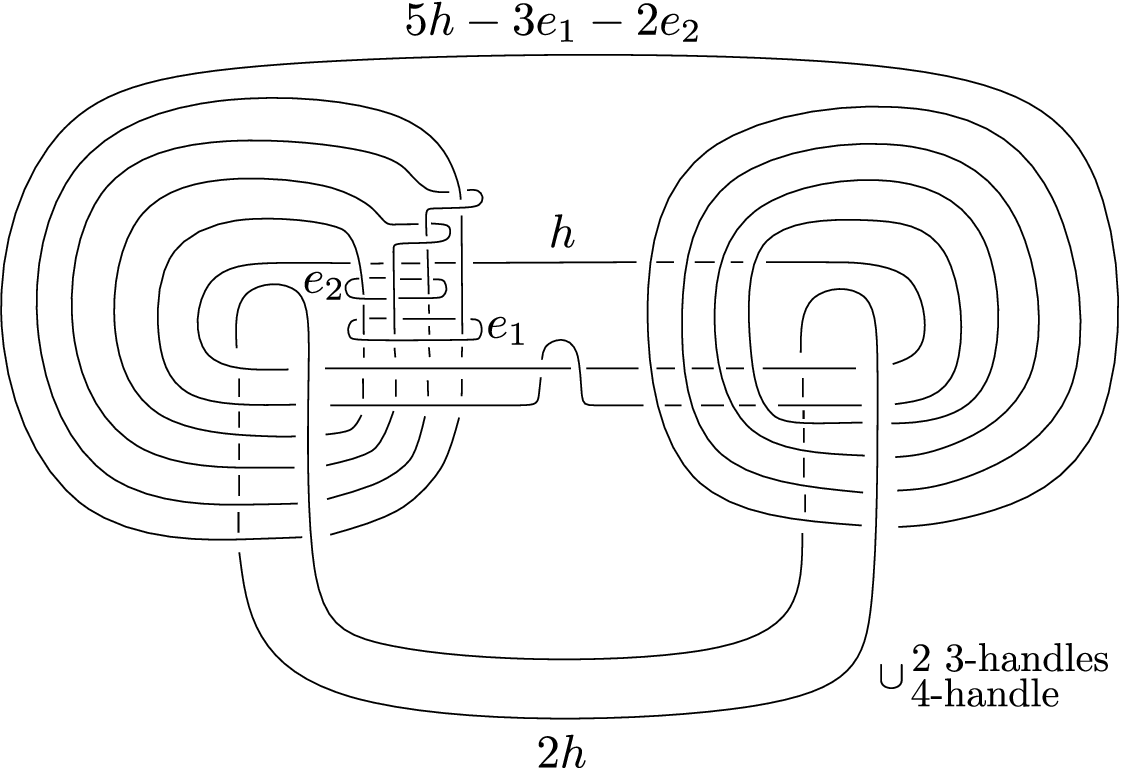}
        \caption{$\mathbf{CP}^2\#2\overline{\mathbf{CP}^2}$.}
        \label{fig:CP2-2CP2_1}
      \end{center}
    \end{figure}
    
    \begin{figure}[htbp]
      \begin{center}
        \includegraphics[width=12cm]{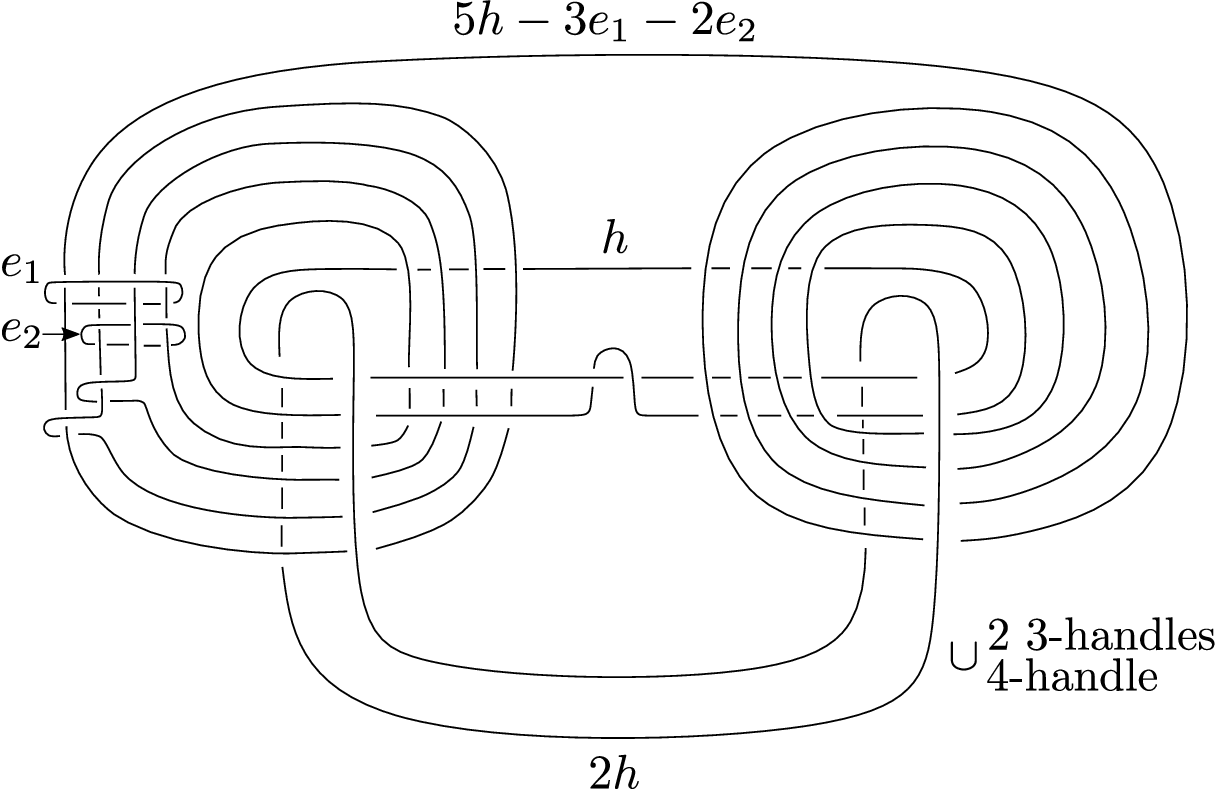}
        \caption{$\mathbf{CP}^2\#2\overline{\mathbf{CP}^2}$.}
        \label{fig:CP2-2CP2_2}
      \end{center}
    \end{figure}
    
    \begin{figure}[htbp]
      \begin{center}
        \includegraphics[width=12cm]{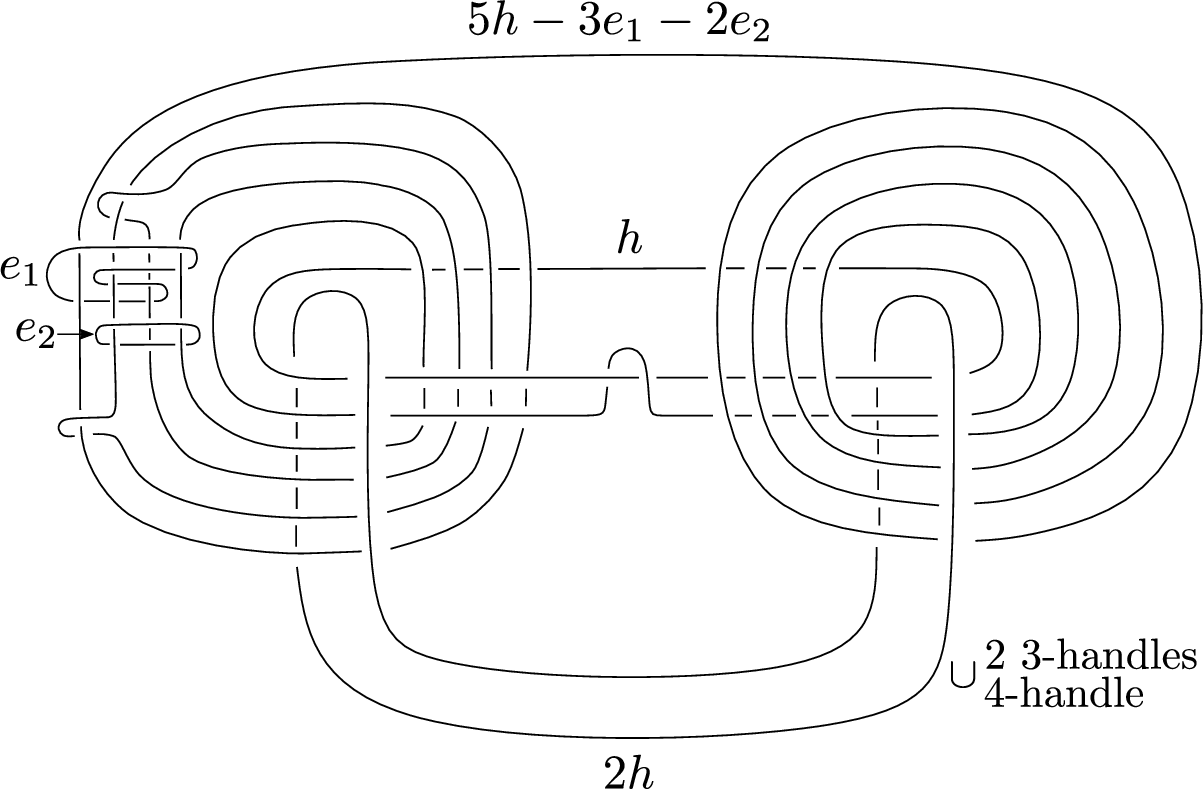}
        \caption{$\mathbf{CP}^2\#2\overline{\mathbf{CP}^2}$.}
        \label{fig:CP2-2CP2_3}
      \end{center}
    \end{figure}
    
    \begin{figure}[htbp]
      \begin{center}
        \includegraphics[width=12cm]{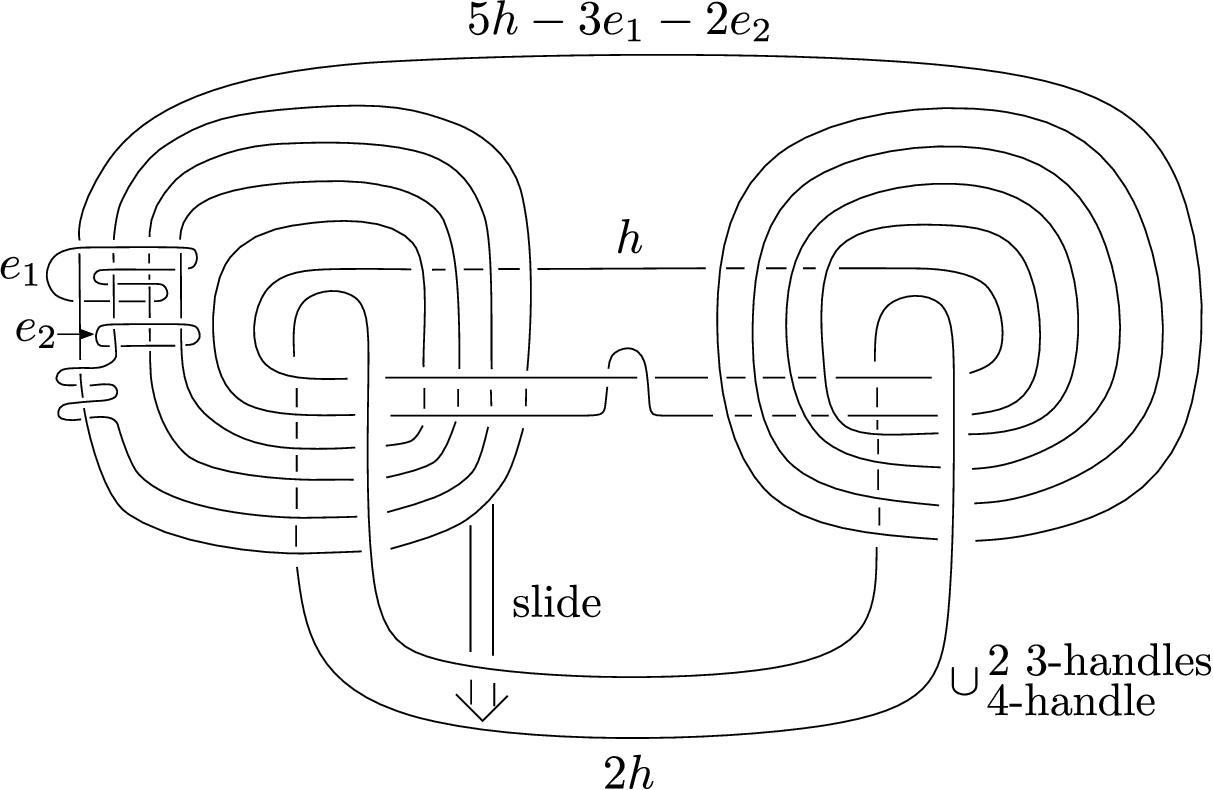}
        \caption{$\mathbf{CP}^2\#2\overline{\mathbf{CP}^2}$.}
        \label{fig:CP2-2CP2_4}
      \end{center}
    \end{figure}
    
    \begin{figure}[htbp]
      \begin{center}
        \includegraphics[width=12cm]{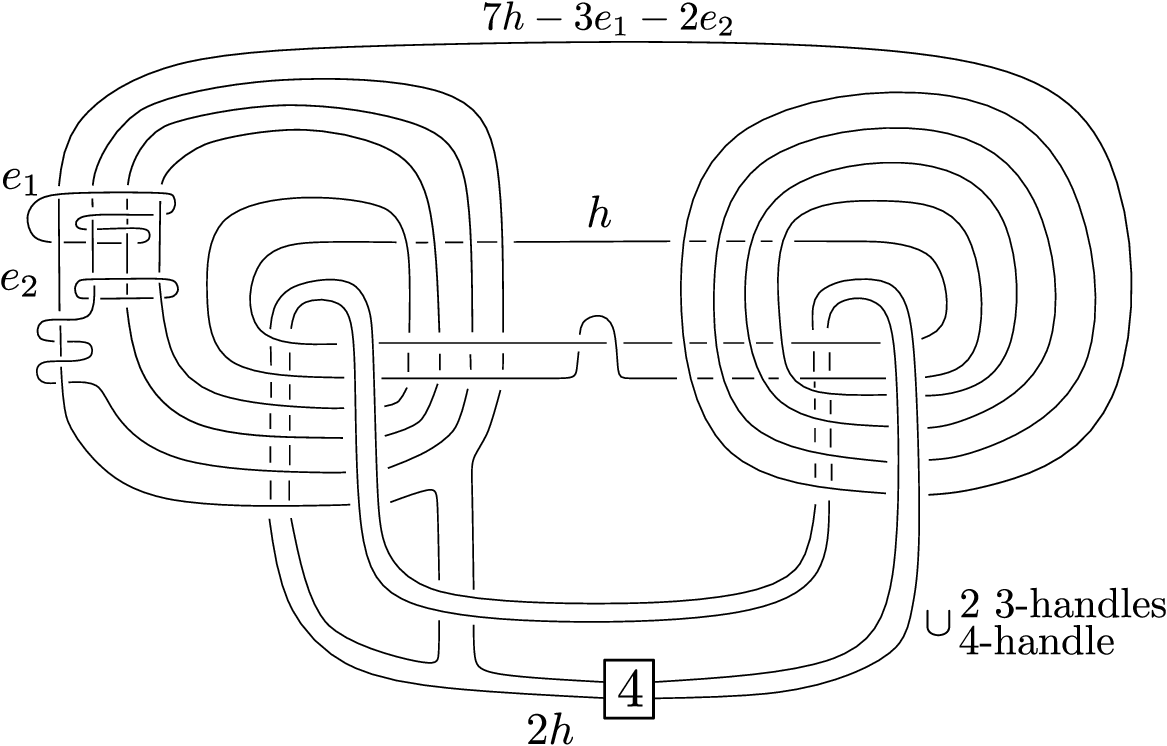}
        \caption{Another diagram of $\mathbf{CP}^2 \# 2\overline{\mathbf{CP}^2}$.}
        \label{fig:CP2-2CP2_5}
      \end{center}
    \end{figure}
    
    \begin{figure}[htbp]
      \begin{center}
        \includegraphics[width=14cm]{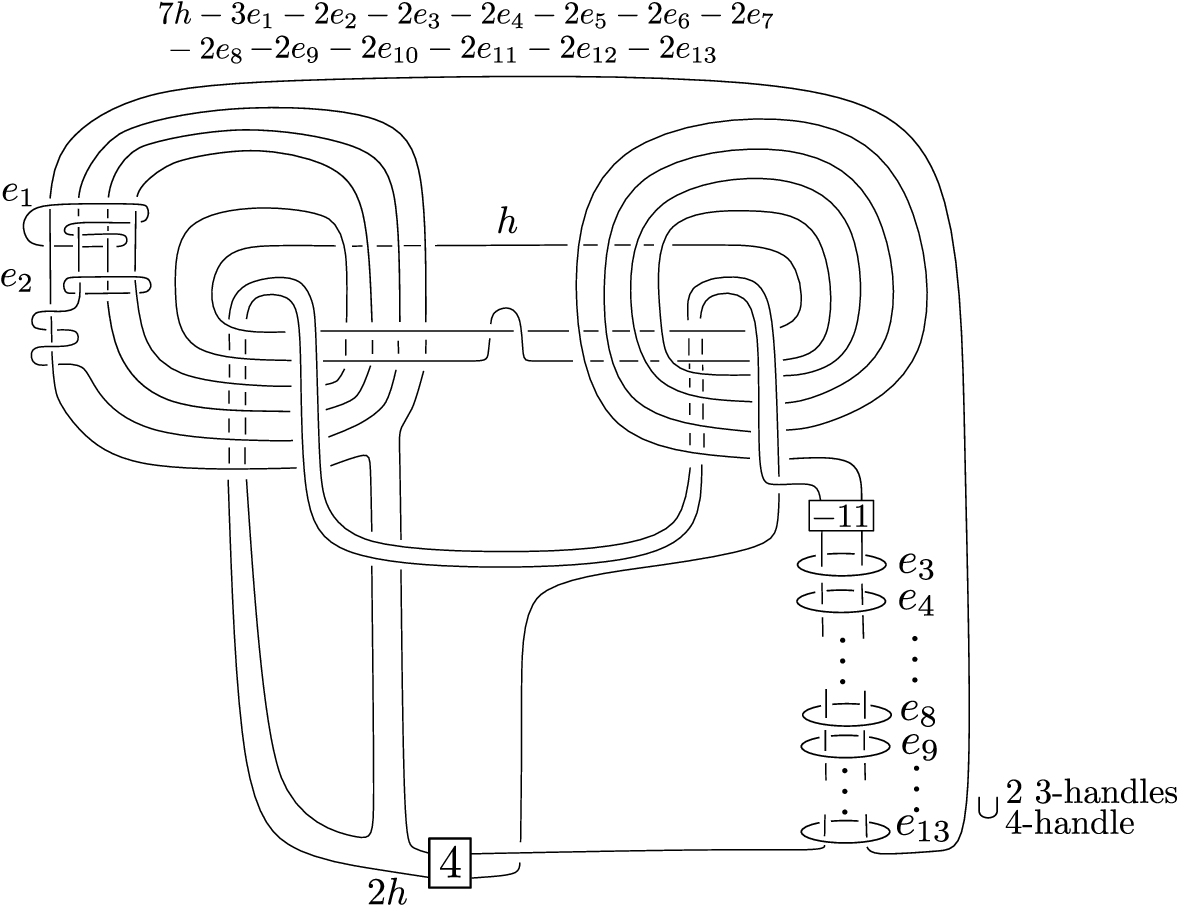}
        \caption{A diagram of $\mathbf{CP}^2 \# 13\overline{\mathbf{CP}^2}$.}
        \label{fig:CP2-13CP2_1}
      \end{center}
    \end{figure}
    
    \begin{figure}[htbp]
      \begin{center}
        \includegraphics[width=14cm]{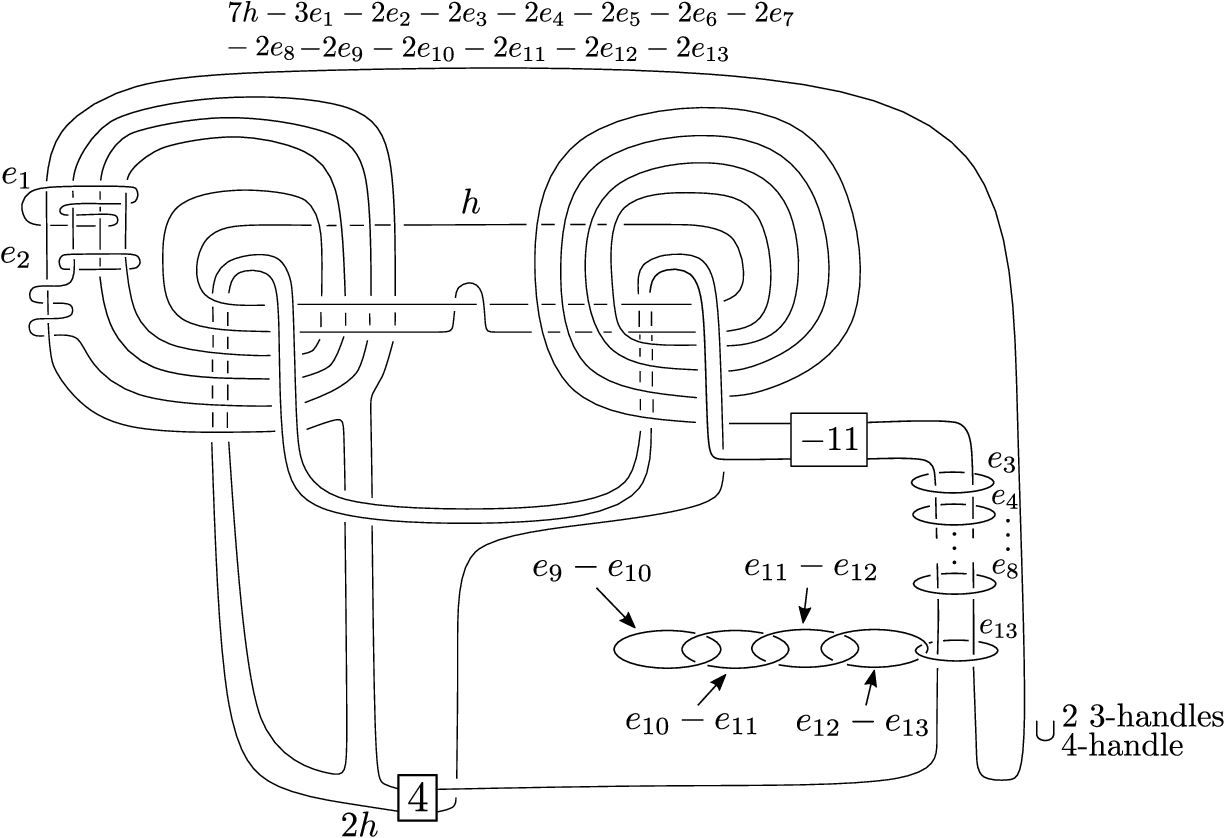}
        \caption{Another diagram of $\mathbf{CP}^2 \# 13\overline{\mathbf{CP}^2}$.}
        \label{fig:CP2-13CP2_2}
      \end{center}
    \end{figure}
  
  \begin{remark}\label{plumbing}
    The reader may notice that there are some differences between the Kirby diagrams in Figure \ref{fig:CP2-14CP2_1} and \cite[Proposition 3.2 (1), $a=4$]{2010Yasui}. First, we kept the attaching circles of all the $2$-handles, whereas in \cite{2010Yasui} the $2$-handles $h, 2h, e_1, e_2, \dots, e_8,$ and $e_{14}$ are omitted for simplicity. Second, we see that Figure \ref{fig:CP2-14CP2_1} we reached is different from the picture as in \cite[Figure 6, $a=4$]{2010Yasui}, even if we remove some attached $2$-handles from Figure \ref{fig:CP2-14CP2_1}. Indeed, none of components $e_8-e_9, e_7-e_8, \cdots , e_3-e_4$ appeared in \cite[Proposition 3.2 (1), $a=4$]{2010Yasui} are contained. This is because we will use two components of $e_3, e_4, \cdots , e_8$ in the proof of Lemma \ref{lem:finding_W2}.
  \end{remark}

  The next corollary follows from the Kirby diagram of $\mathbf{CP}^2 \# 14\overline{\mathbf{CP}^2}$ in Figure \ref{fig:CP2-14CP2_1}.
  
  \begin{corollary}[{\cite[Corollary 3.3 (1), $a=4$]{2010Yasui}}]\label{cor:C_7}
    $\mathbf{CP}^2 \# 14\overline{\mathbf{CP}^2}$ contains a copy of $C_7$ such that the elements $u_1,u_2,\dots, u_{6}$ of $H_2(C_7;\mathbf{Z})$ in $H_2(\mathbf{CP}^2 \# 14\overline{\mathbf{CP}^2};\mathbf{Z})$ satisfy
    \begin{align}
      &u_{i} = e_{8+i} - e_{9+i} ~ (1\leq i \leq 5), \nonumber\\
      &u_6 = 7h-3e_1 -2e_2 -2e_3 - \dots -2e_{13}-e_{14}.\nonumber
    \end{align}
  \end{corollary}
  
  \begin{proof}
    In the bottom right part of Figure \ref{fig:CP2-14CP2_1}, we can find a part of the diagram of $C_7$. Their homology classes satisfy the above equations. We can also find the $2$-handle $u_6$, so what we need to check is that the attaching circle of $u_6 = 7h-3e_1-2e_2-2e_3-2e_4-2e_5-2e_6-2e_7-2e_8-2e_9-2e_{10}-2e_{11}-2e_{12}-2e_{13}-e_{14}$ is the unknot. If we delete all the other $2$-handles from Figure \ref{fig:CP2-14CP2_1}, it looks like Figure \ref{fig:attaching_circle_1}. We can isotope this diagram to obtain Figure \ref{fig:attaching_circle_2}. We can find four positive full twists on the right side of Figure \ref{fig:attaching_circle_2}. By an isotopy, we obtain Figure \ref{fig:attaching_circle_3}. We can find another four positive full twists from the left side of Figure \ref{fig:attaching_circle_3}, and we obtain Figure \ref{fig:attaching_circle_4} by canceling those full twists again. The reader can check that this diagram represents the unknot.
  \end{proof}
    
    \begin{figure}[htbp]
      \begin{center}
        \includegraphics[width=10cm]{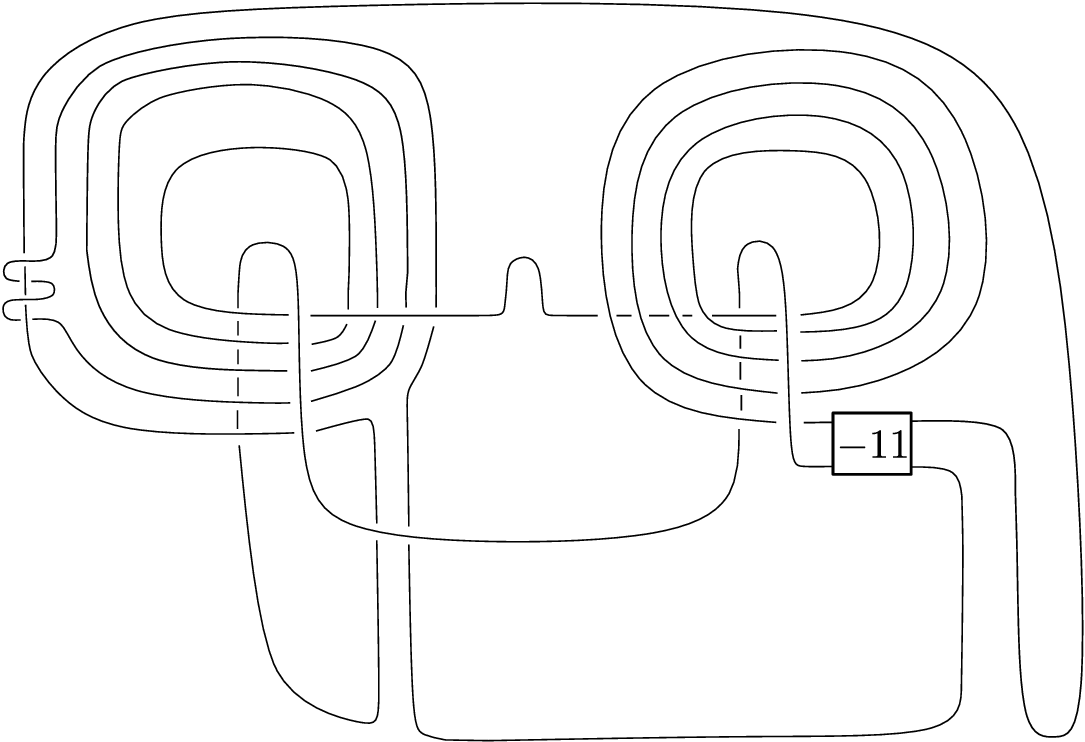}
        \caption{The attaching circle of $7h-3e_1-2e_2-2e_3-2e_4-2e_5-2e_6-2e_7-2e_8-2e_9-2e_{10}-2e_{11}-2e_{12}-2e_{13}-e_{14}$ in Figure \ref{fig:CP2-14CP2_1}.}
        \label{fig:attaching_circle_1}
      \end{center}
    \end{figure}
    
    \begin{figure}[htbp]
      \begin{center}
        \includegraphics[width=10cm]{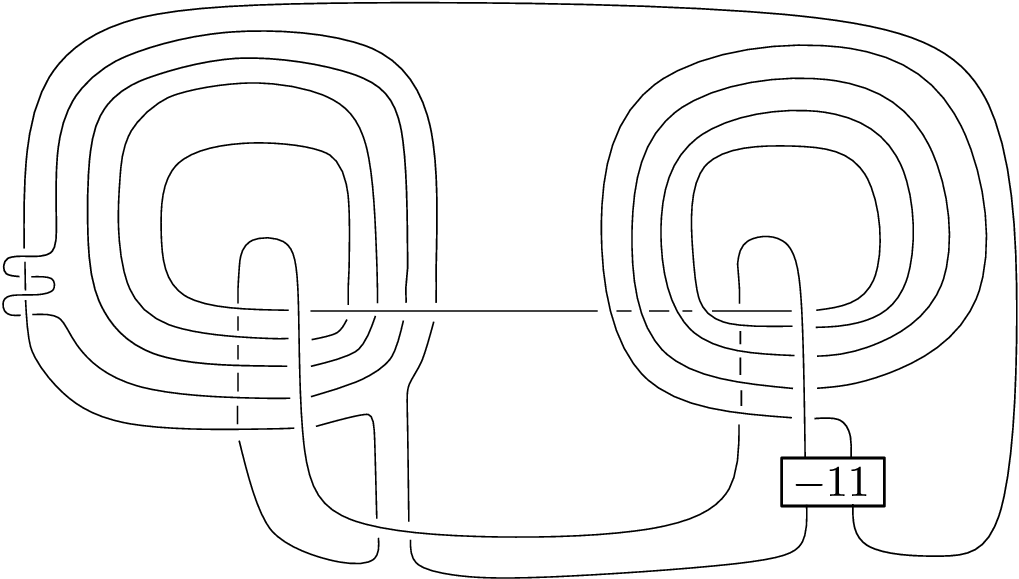}
        \caption{The attaching circle of $7h-3e_1-2e_2-2e_3-2e_4-2e_5-2e_6-2e_7-2e_8-2e_9-2e_{10}-2e_{11}-2e_{12}-2e_{13}-e_{14}$.}
        \label{fig:attaching_circle_2}
      \end{center}
    \end{figure}
    
    \begin{figure}[htbp]
      \begin{center}
        \includegraphics[width=10cm]{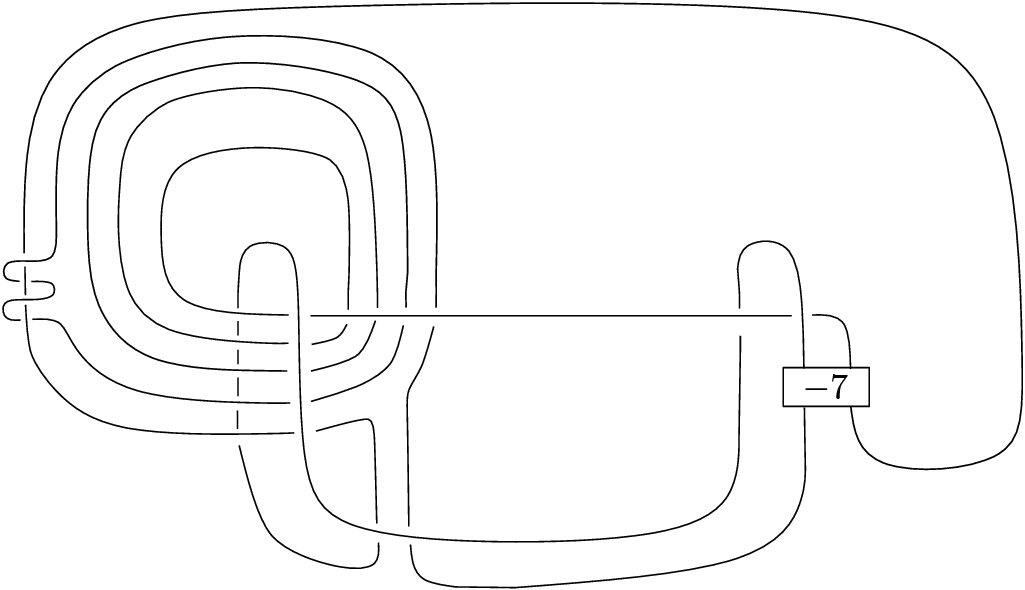}
        \caption{The attaching circle of $7h-3e_1-2e_2-2e_3-2e_4-2e_5-2e_6-2e_7-2e_8-2e_9-2e_{10}-2e_{11}-2e_{12}-2e_{13}-e_{14}$.}
        \label{fig:attaching_circle_3}
      \end{center}
    \end{figure}
    
    \begin{figure}[htbp]
      \begin{center}
        \includegraphics[width=10cm]{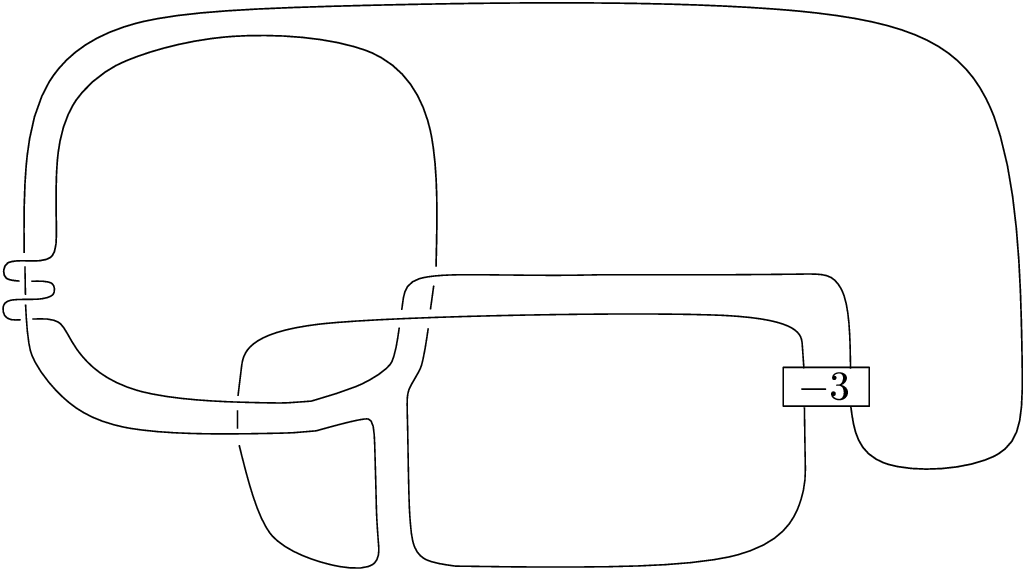}
        \caption{The attaching circle of $7h-3e_1-2e_2-2e_3-2e_4-2e_5-2e_6-2e_7-2e_8-2e_9-2e_{10}-2e_{11}-2e_{12}-2e_{13}-e_{14}$.}
        \label{fig:attaching_circle_4}
      \end{center}
    \end{figure}
  
  \begin{definition}[{\cite[Definition 3.5, $a=4$]{2010Yasui}}]\label{def:R8}
    The rational blowdown of $\mathbf{CP}^2 \allowbreak \# 14 \overline{\mathbf{CP}^2}$ along the embedded copy of $C_7$ in Corollary \ref{cor:C_7} is denoted as $R_8$.
  \end{definition}
  
  \begin{remark}
    Hereafter, we will not need to track the information of the homology classes of $2$-handles. So, we will use framing coefficients to represent the framings of $2$-handles in the following diagrams. We obtain Figure \ref{fig:CP2-14CP2_2} by squaring the homology classes in Figure \ref{fig:CP2-14CP2_1}.
    \end{remark}
  
    \begin{figure}[htbp]
      \begin{center}
        \includegraphics[width=12cm]{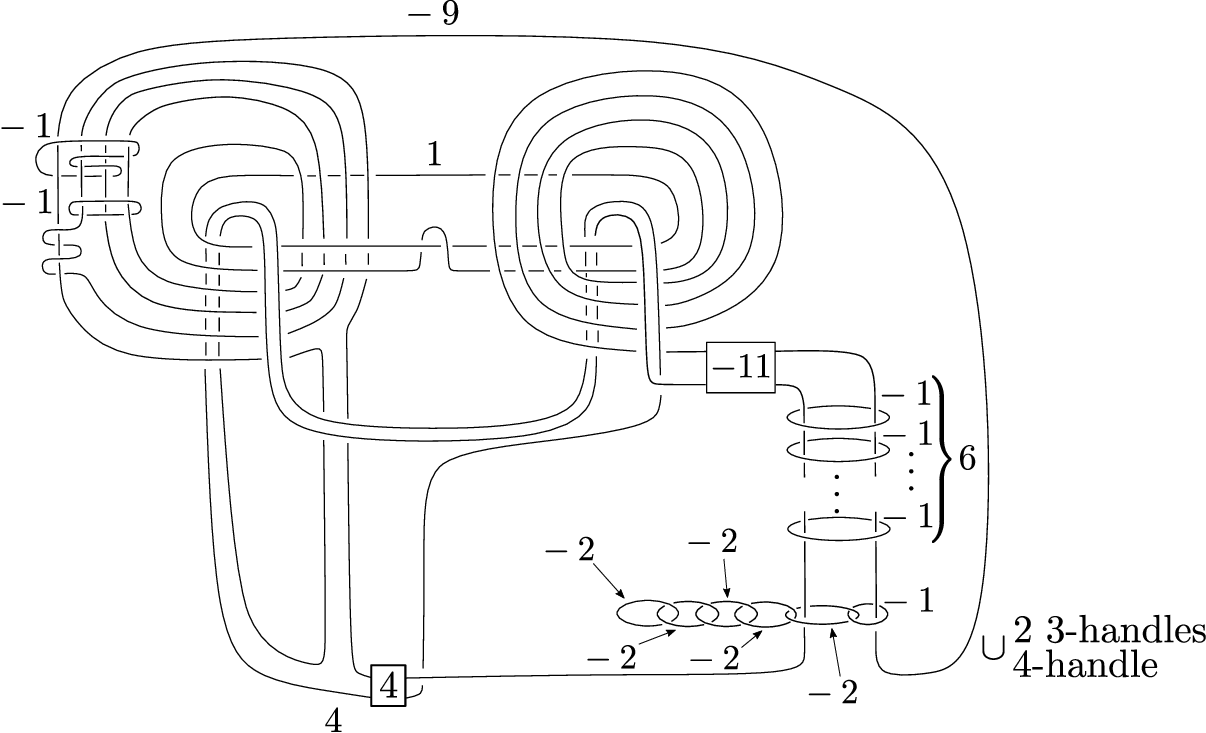}
        \caption{A diagram of $\mathbf{CP}^2 \# 14\overline{\mathbf{CP}^2}$.}
        \label{fig:CP2-14CP2_2}
      \end{center}
    \end{figure}
    
    Now we will draw a diagram of $R_8$. 
    
    \begin{theorem}\label{thm:diagram_of_R8}
      The $4$-manifold $R_8$ admits the handle decomposition in Figure \ref{fig:R8_1}.
      \begin{figure}[htbp]
        \begin{center}
          \includegraphics[width=12cm]{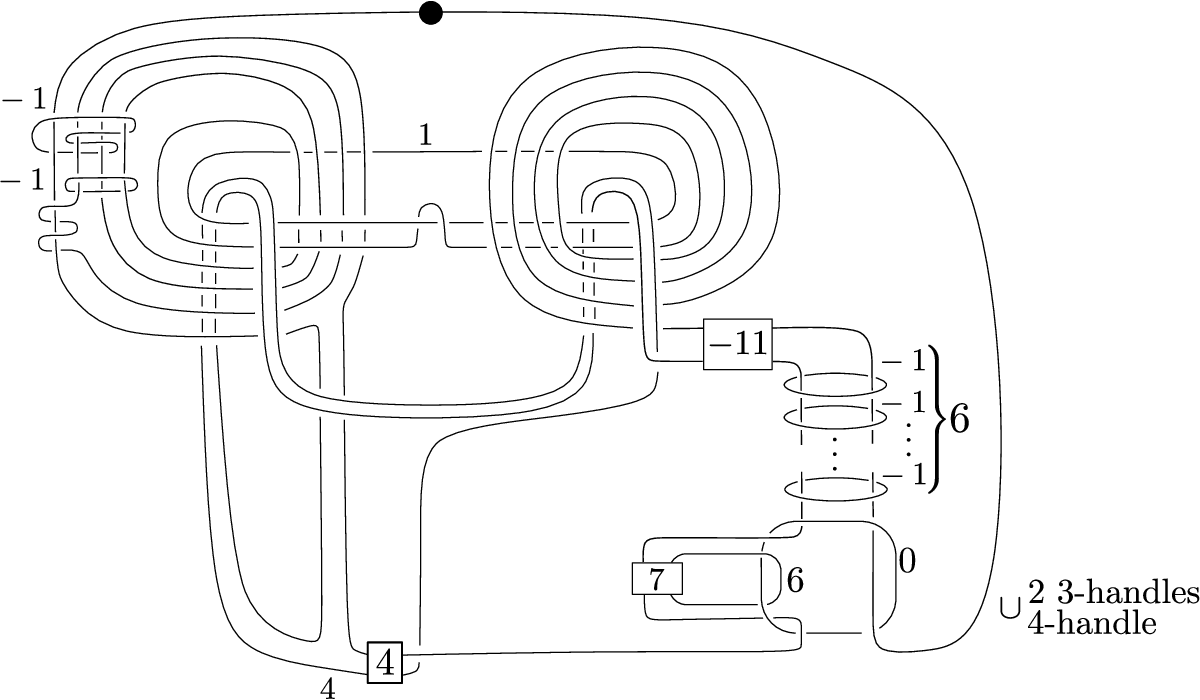}
          \caption{A diagram of $R_8$.}
          \label{fig:R8_1}
        \end{center}
      \end{figure}
    \end{theorem}
    
    \begin{proof}
      Recall that there is a procedure to draw a diagram of a rational blowdown \cite[Figure 16]{2008Akbulut-Yasui}. If we apply this procedure to the copy of $C_7$ in Figure \ref{fig:CP2-14CP2_1}, we obtain Figure \ref{fig:R8_1}.
    \end{proof}

  \begin{remark}\label{rmk:C_7}
    In \cite[Proposition 3.9]{2010Yasui}, Yasui showed that a $4$-manifold obtained by his construction is homeomorphic to $\mathbf{CP}^2 \# 8\overline{\mathbf{CP}^2}$. To prove that it is homeomorphic to $\mathbf{CP}^2 \# 8\overline{\mathbf{CP}^2}$, he checked the simply-connectedness of his manifold to apply Rochlin's theorem and Freedman's theorem. He also proved that a $4$-manifold obtained by his construction is not diffeomorphic to $\mathbf{CP}^2 \# 8\overline{\mathbf{CP}^2}$ by showing that its Seiberg-Witten invariant is non-trivial \cite[Lemma 5.1 (1)]{2010Yasui}. To prove the non-triviality of the Seiberg-Witten invariant, he first calculated a non-trivial value of the Seiberg-Witten invariant of $\mathbf{CP}^2 \# 14\overline{\mathbf{CP}^2}$. Then he used the information of the homology classes of the $2$-handles of $C_7$ to apply the theorems of Fintushel and Stern \cite{1997Fintushel-Stern} on the Seiberg-Witten invariant of the rational blowdown, and proved that a value of the Seiberg-Witten invariant of his manifold coincides with the above non-trivial value of the Seiberg-Witten invariant of $\mathbf{CP}^2 \# 14\overline{\mathbf{CP}^2}$. 
    
    We followed Yasui's procedure \cite{2010Yasui} to construct $C_7$ in $\mathbf{CP}^2\allowbreak \# \allowbreak 14\overline{\mathbf{CP}^2}$ with small modifications. However, we can still apply the same arguments to prove that $R_8$ is an exotic $\mathbf{CP}^2 \# 8\overline{\mathbf{CP}^2}$. First, by sliding a $-1$-framed unknot over the $0$-framed unknot in the diagram in Figure \ref{fig:R8_1}, we obtain a diagram of $R_8$ where we can cancel the unique $1$-handle. So $R_8$ is simply-connected, and one can check that $R_8$ is homeomorphic to $\mathbf{CP}^2 \# 8\overline{\mathbf{CP}^2}$ by applying Rochlin's theorem and Freedman's theorem as in Yasui's argument. Also, the homology classes of the $2$-handles of $C_7$ in Figure \ref{fig:CP2-14CP2_1} and those in \cite[Corollary 3.3 (1)]{2010Yasui} are the same. Therefore, the non-triviality of the Seiberg-Witten invariant of $R_8$ has also been proved by Yasui's argument. Thus, we obtain the following theorem.
  \end{remark}

  \begin{theorem}[cf. {\cite[Corollary 5.2 (1)]{2010Yasui}}]\label{thm:R8}
    The $4$-manifold $R_8$ is homeomorphic but not diffeomorphic to $\mathbf{CP}^2 \# 8\overline{\mathbf{CP}^2}$.
  \end{theorem}

\section{Finding a cork}\label{section:finding_a_cork}
  
  In this section, we prove Theorem \ref{thm:cork} and Corollary \ref{cor:stab}. To prove Theorem \ref{thm:cork}, we find an embedding of a cork into $R_8$ by using the diagram in Corollary \ref{thm:diagram_of_R8}, and show that the cork twist of $R_8$ along this cork results in the standard $\mathbf{CP}^2 \# 8\overline{\mathbf{CP}^2}$. First, let us briefly review the basic terminology associated with corks.

  \begin{definition}\label{def:cork}
    Let $(C,\tau)$ be a pair of a compact contractible $4$-manifold with boundary and an involution $\tau$ on the boundary $\partial C$. We call $(C,\tau)$ a \textit{cork} if $\tau$ does not extend to any self-diffeomorphism on $C$. Let $i$ be an embedding of $C$ into a $4$-manifold $X$. The $4$-manifold $X_{(C,\tau, i)}:= (X - \textrm{int}(i(C))) \cup_{i\circ\tau} C$ is called the \textit{cork twist of $X$ along $(C,\tau, i)$}. We say $i$ is an \textit{effective embedding of $(C, \tau)$ into $X$} if $X_{(C,\tau, i)}$ is not diffeomorphic to $X$. Note that for any cork $(C,\tau)$ and any embedding $i$ of $C$ into X, the cork twist $X_{(C,\tau, i)}$ is always homeomorphic to $X$ by Freedman's result \cite{1982Freedman}. A cork $(C,\tau)$ is called a \textit{cork of} a $4$-manifold $X$ if there exists an effective embedding of $(C, \tau)$ into $X$.
  \end{definition}

  \begin{lemma}\label{lem:finding_W2}
    \textup{(1)} The $4$-manifold $R_8$ admits the handle decomposition in Figure \ref{fig:R8_2}.\\
    \textup{(2)} The handle decomposition in Figure \ref{fig:R8_2} contains a subhandlebody diffeomorphic to $W_2$.
  \end{lemma}

  \begin{proof}
    (1) Figure \ref{fig:R8_local_1} (a) shows a local part of the diagram of $R_8$ in Figure \ref{fig:R8_1}. For the following argument, we often draw the local parts of the diagrams unless necessary. As mentioned in Section \ref{convention}, the parts of the diagrams not drawn in the figures are always fixed. We slide $-1$-framed unknot over the $0$-framed unknot, and then we isotope the diagram to obtain Figure \ref{fig:R8_local_1} (b). We can isotope this diagram to Figure \ref{fig:R8_local_1} (c). We slide the $0$-framed unknot over the $6$-framed unknot to obtain Figure \ref{fig:R8_local_1} (d). By isotopies, we obtain Figure \ref{fig:R8_local_1} (e), (f), and then Figure \ref{fig:R8_local_2} (a). We slide the $4$-framed unknot over the $-1$-framed unknot and then isotope the diagram to obtain Figure \ref{fig:R8_local_2} (b). By another isotopy, we obtain Figure \ref{fig:R8_local_2} (c). By sliding the $0$-framed unknot in the diagram over one of the $-1$-unknots, we obtain Figure \ref{fig:R8_local_2} (d). Now the entire diagram looks like Figure \ref{fig:R8_2}. Thus, $R_8$ admits the handle decomposition in Figure \ref{fig:R8_2}.
    
    (2) The handle decomposition in Figure \ref{fig:R8_2} contains the subhandlebody described in Figure \ref{fig:W2_in_R8_1}. By isotoping this diagram likewise in Figure \ref{fig:attaching_circle_1}, \ref{fig:attaching_circle_2}, \ref{fig:attaching_circle_3}, and \ref{fig:attaching_circle_4}, we obtain Figure \ref{fig:W2_in_R8_2} (a). We isotope this diagram to obtain Figure \ref{fig:W2_in_R8_2} (b). By another isotopy, we can deform this diagram to Figure \ref{fig:W2_in_R8_2} (c), a diagram of $W_2$. Therefore, the handle decomposition in Figure \ref{fig:R8_2} contains a subhandlebody diffeomorphic to $W_2$.
  \end{proof}

    \begin{figure}[htbp]
      \begin{center}
        \includegraphics[width=12cm]{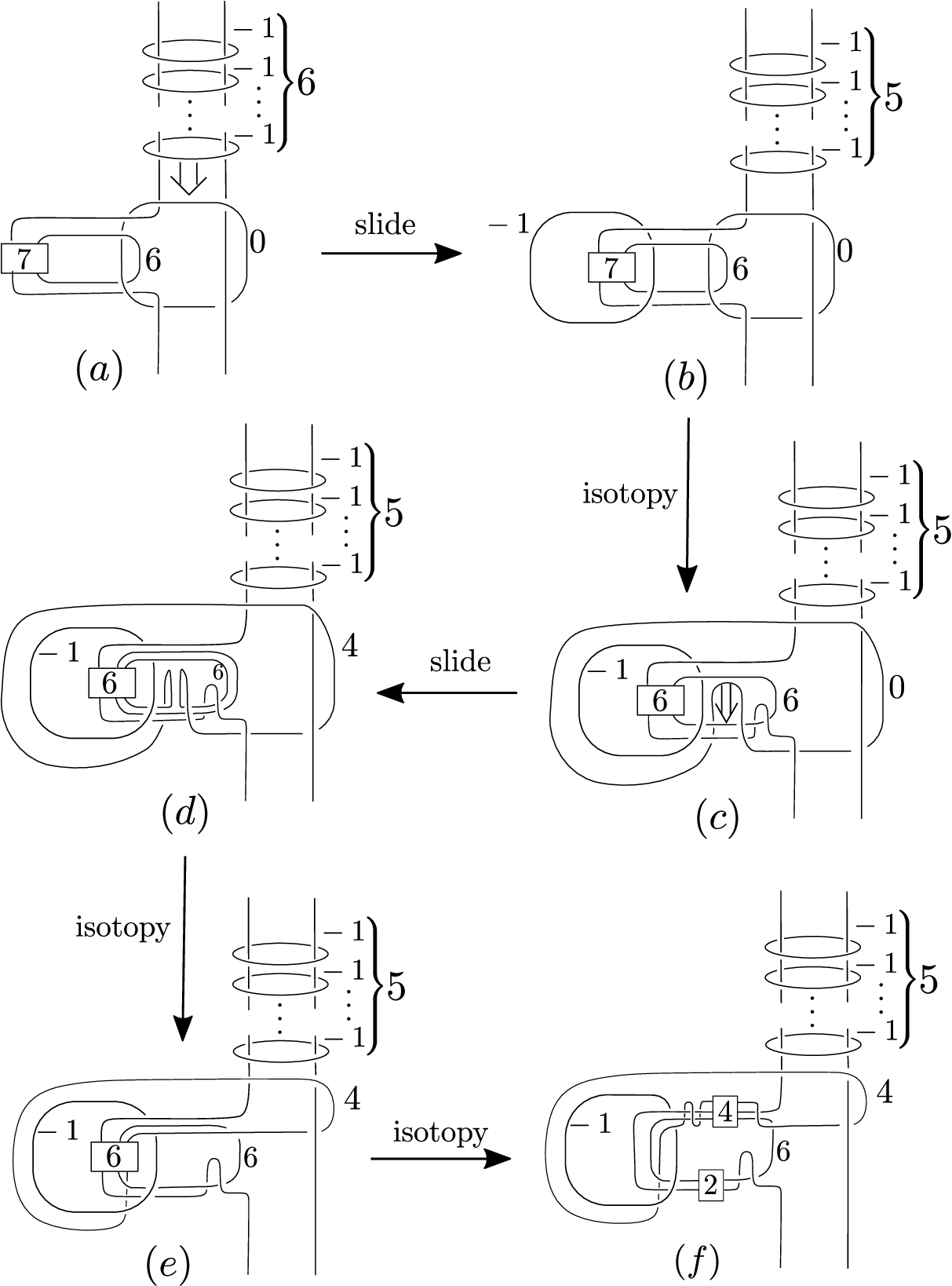}
        \caption{Local deformations of Figure \ref{fig:R8_1}.}
        \label{fig:R8_local_1}
      \end{center}
    \end{figure}
    
    \begin{figure}[htbp]
      \begin{center}
        \includegraphics[width=12cm]{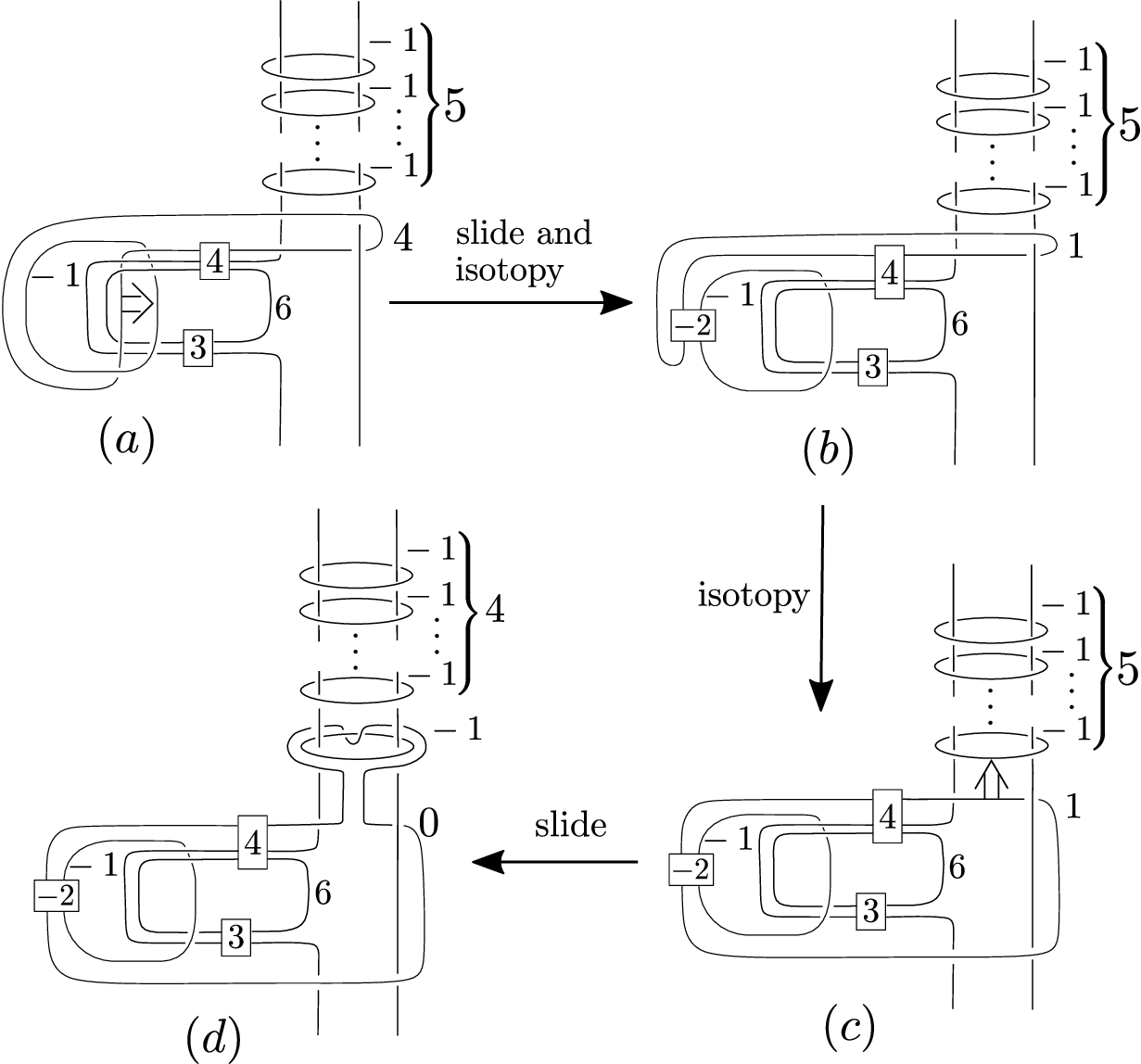}
        \caption{Local deformations of Figure \ref{fig:R8_1}.}
        \label{fig:R8_local_2}
      \end{center}
    \end{figure}
    
    \begin{figure}[htbp]
      \begin{center}
        \includegraphics[width=13cm]{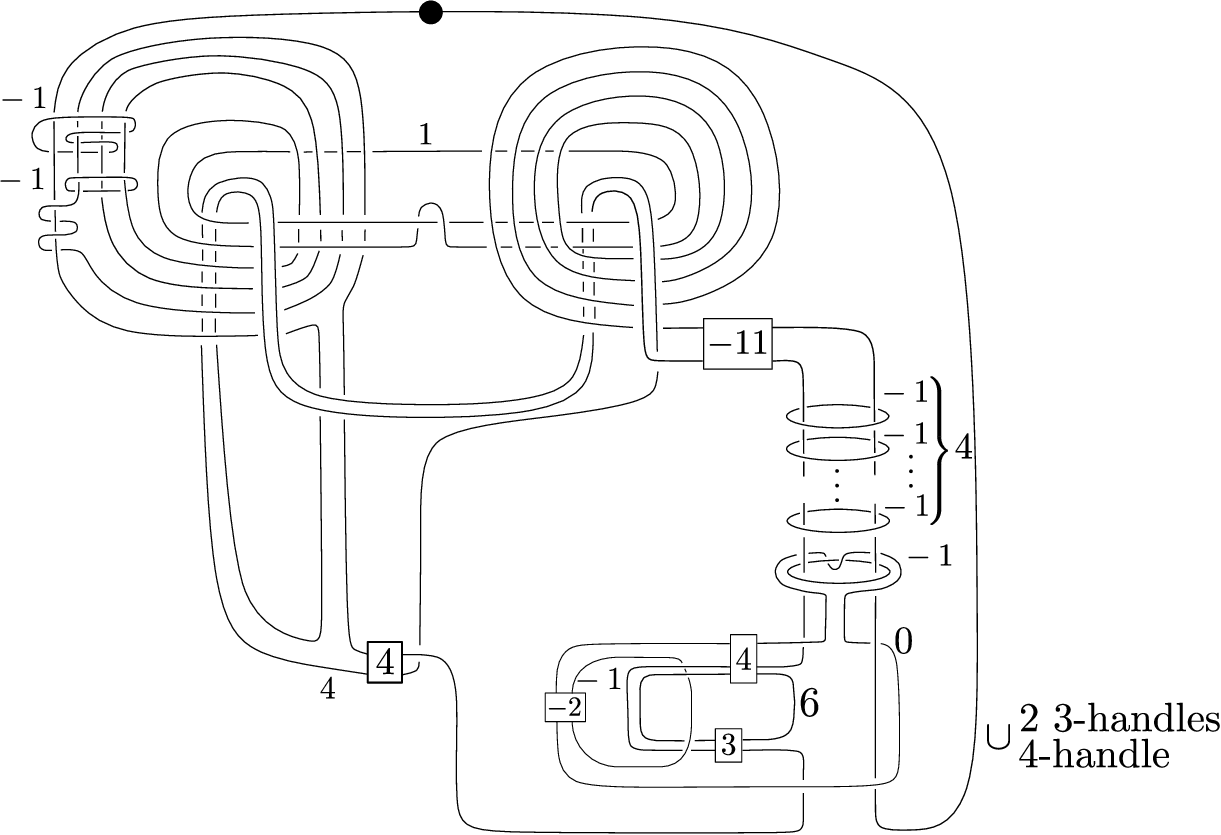}
        \caption{Another diagram of $R_8$.}
        \label{fig:R8_2}
      \end{center}
    \end{figure}
    
    \begin{figure}[htbp]
      \begin{center}
        \includegraphics[width=10cm]{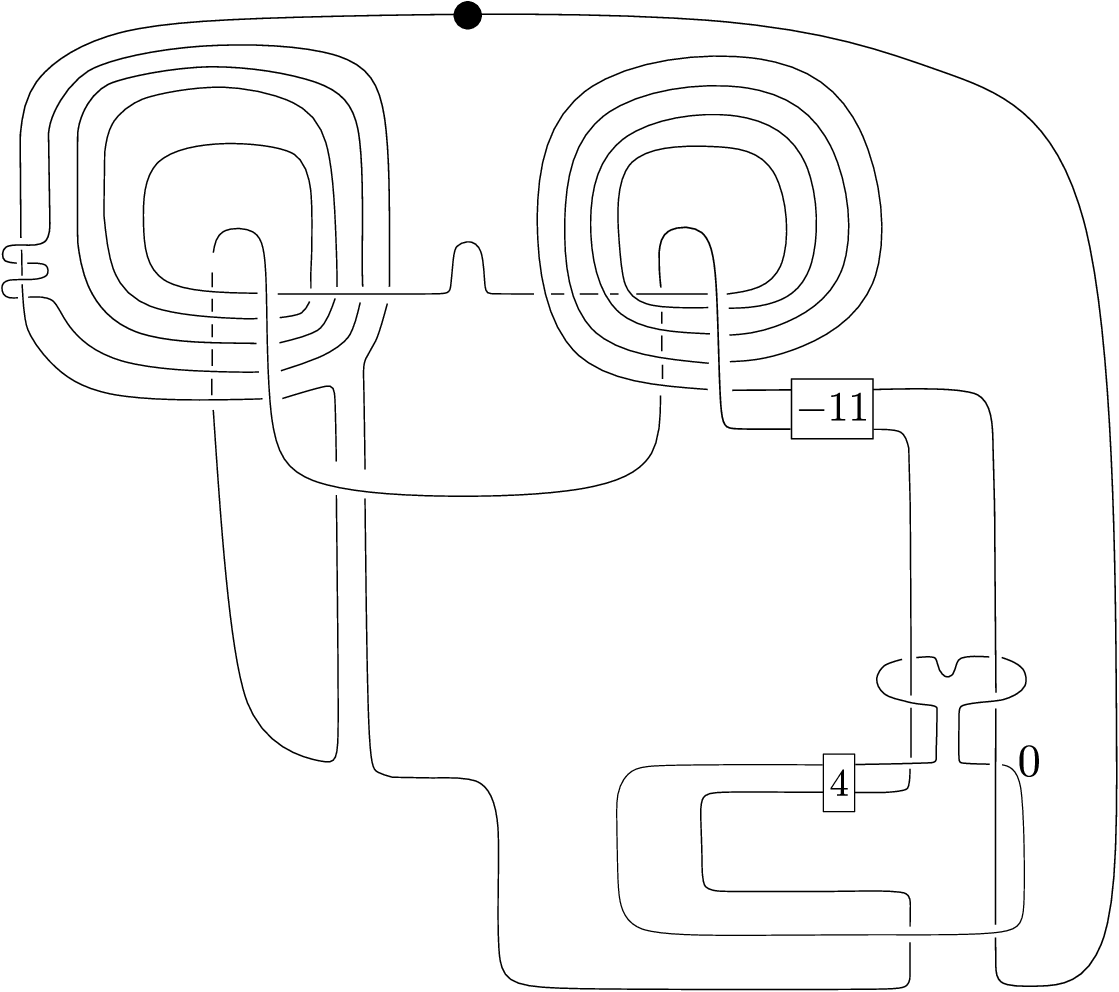}
        \caption{A subhandlebody contained in the handle decomposition in Figure \ref{fig:R8_2}.}
        \label{fig:W2_in_R8_1}
      \end{center}
    \end{figure}
    
    \begin{figure}[htbp]
      \begin{center}
        \includegraphics[width=13cm]{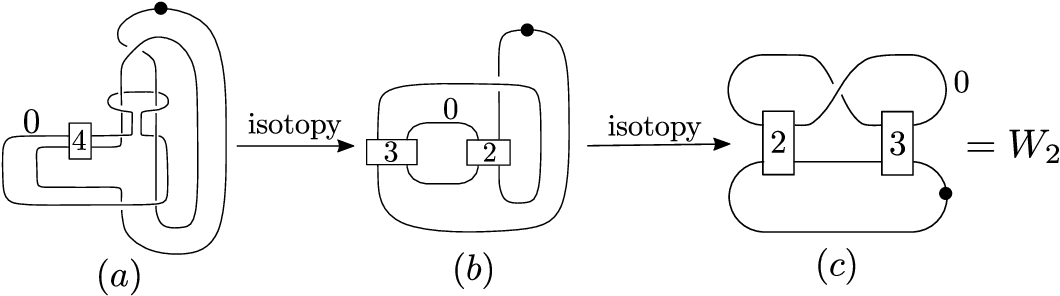}
        \caption{Isotopies.}
        \label{fig:W2_in_R8_2}
      \end{center}
    \end{figure}

  Now we are ready to prove Theorem \ref{thm:cork}.
  
  \begin{proof}[Proof of Theorem \ref{thm:cork}]
    From Lemma \ref{lem:finding_W2}, we see that there exists an embedding of $W_2$ into $R_8$. We call this embedding $i$. It is enough to show that the cork twist of $R_8$ along $(W_2,f_2, i)$ results in $\mathbf{CP}^2 \# 8\overline{\mathbf{CP}^2}$. Figure \ref{fig:R8_corktwist} shows a diagram of the cork twist of $R_8$ along $(W_2,f_2, i)$. Figure \ref{fig:R8_corktwist_local_1} (a) shows the bottom right part of this diagram. By following the steps described in Figure \ref{fig:R8_corktwist_local_1}, we obtain Figure \ref{fig:R8_corktwist_local_1} (f). After performing a handle slide as shown in this figure, we can isotope the diagram to obtain Figure \ref{fig:R8_corktwist_local_2} (a). By following the steps described in Figure \ref{fig:R8_corktwist_local_2}, we obtain Figure \ref{fig:R8_corktwist_local_2} (f). Note that the deformation from Figure \ref{fig:R8_corktwist_local_2} (b) to Figure \ref{fig:R8_corktwist_local_2} (c), in which the dot and $0$ are exchanged, gives a diffeomorphism because the $0$-framed unknot geometrically links with the dotted circle once, so this exchange represents cutting out a $4$-ball and pasting it back. After a handle slide and isotopy, we obtain Figure \ref{fig:R8_corktwist_local_3} (a). By following the steps described in Figure \ref{fig:R8_corktwist_local_3}, we obtain Figure \ref{fig:R8_corktwist_local_3} (f). After a handle slide and isotopy, we obtain Figure \ref{fig:R8_corktwist_local_4} (a). By following the steps described in Figure \ref{fig:R8_corktwist_local_4}, we obtain Figure \ref{fig:R8_corktwist_local_4} (d). The whole Kirby diagram is shown in Figure \ref{fig:CP2-8CP2_1}, and we obtain Figure \ref{fig:CP2-8CP2_2} by an isotopy. Figure \ref{fig:CP2-8CP2_2} represents $\mathbf{CP}^2 \# 8\overline{\mathbf{CP}^2}$ since we can obtain this diagram after performing $6$ blowups in Figure \ref{fig:CP2-2CP2_5} in the same manner we obtained Figure \ref{fig:CP2-13CP2_1}.
  \end{proof}
    \begin{figure}[htbp]
      \begin{center}
        \includegraphics[width=14cm]{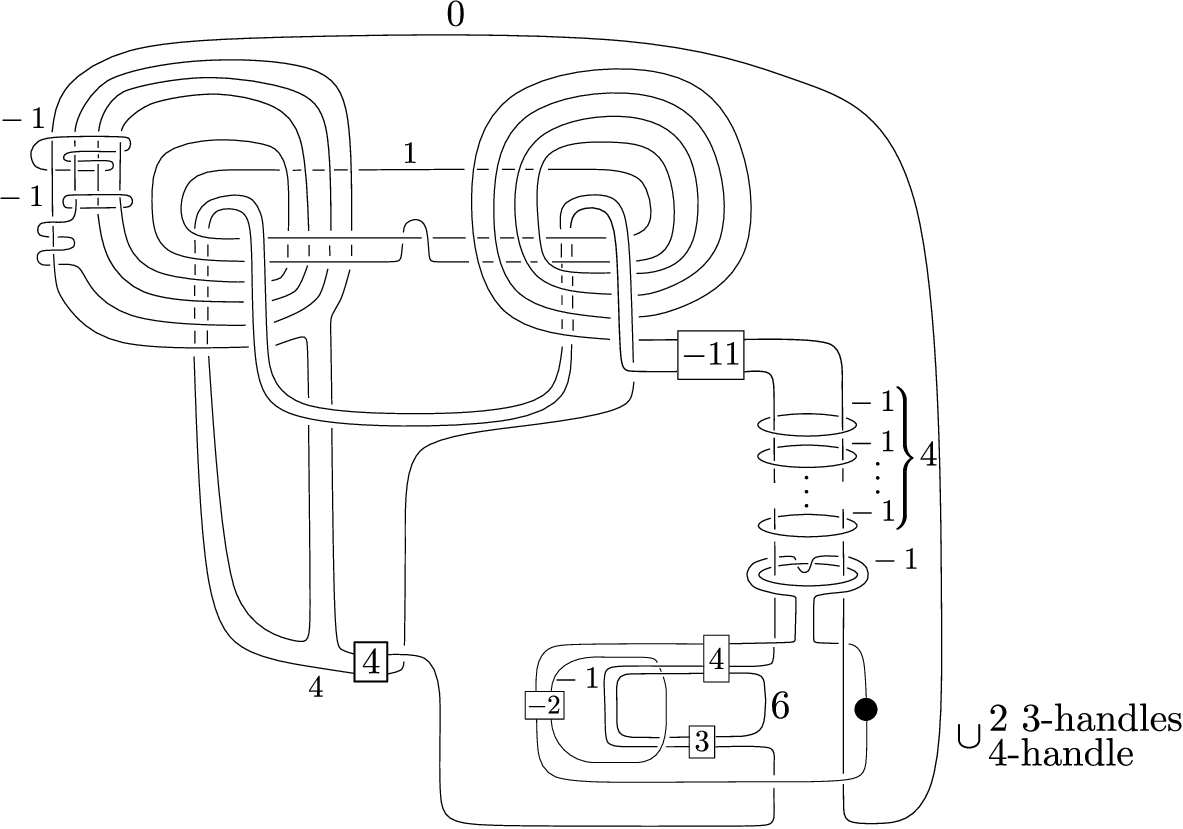}
        \caption{The cork twist of $R_8$ along $(W_2,f_2)$.}
        \label{fig:R8_corktwist}
      \end{center}
    \end{figure}
    \begin{figure}[htbp]
      \begin{center}
        \includegraphics[width=11cm]{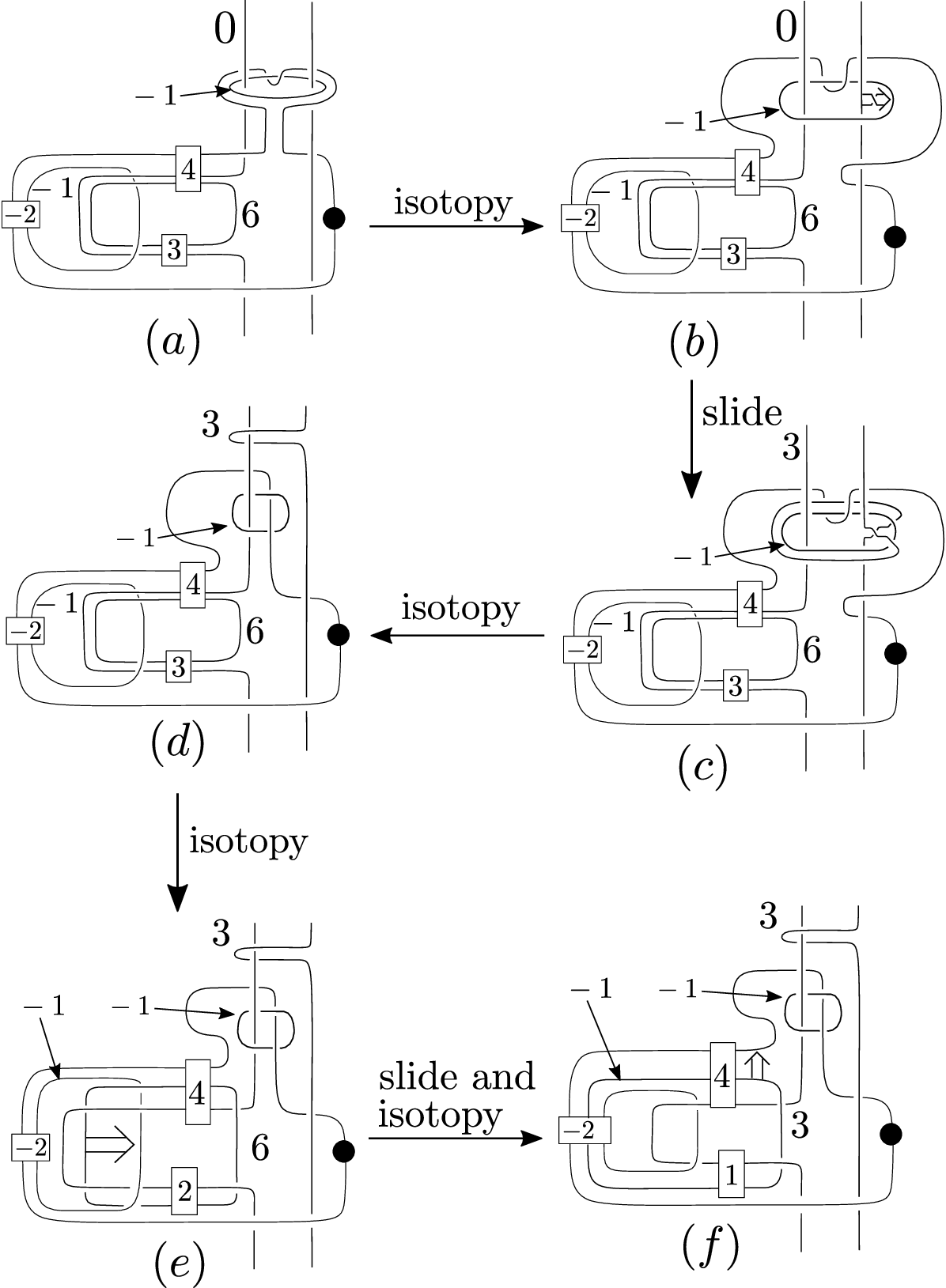}
        \caption{Local deformations of Figure \ref{fig:R8_corktwist}.}
        \label{fig:R8_corktwist_local_1}
      \end{center}
    \end{figure}
    \begin{figure}[htbp]
      \begin{center}
        \includegraphics[width=11cm]{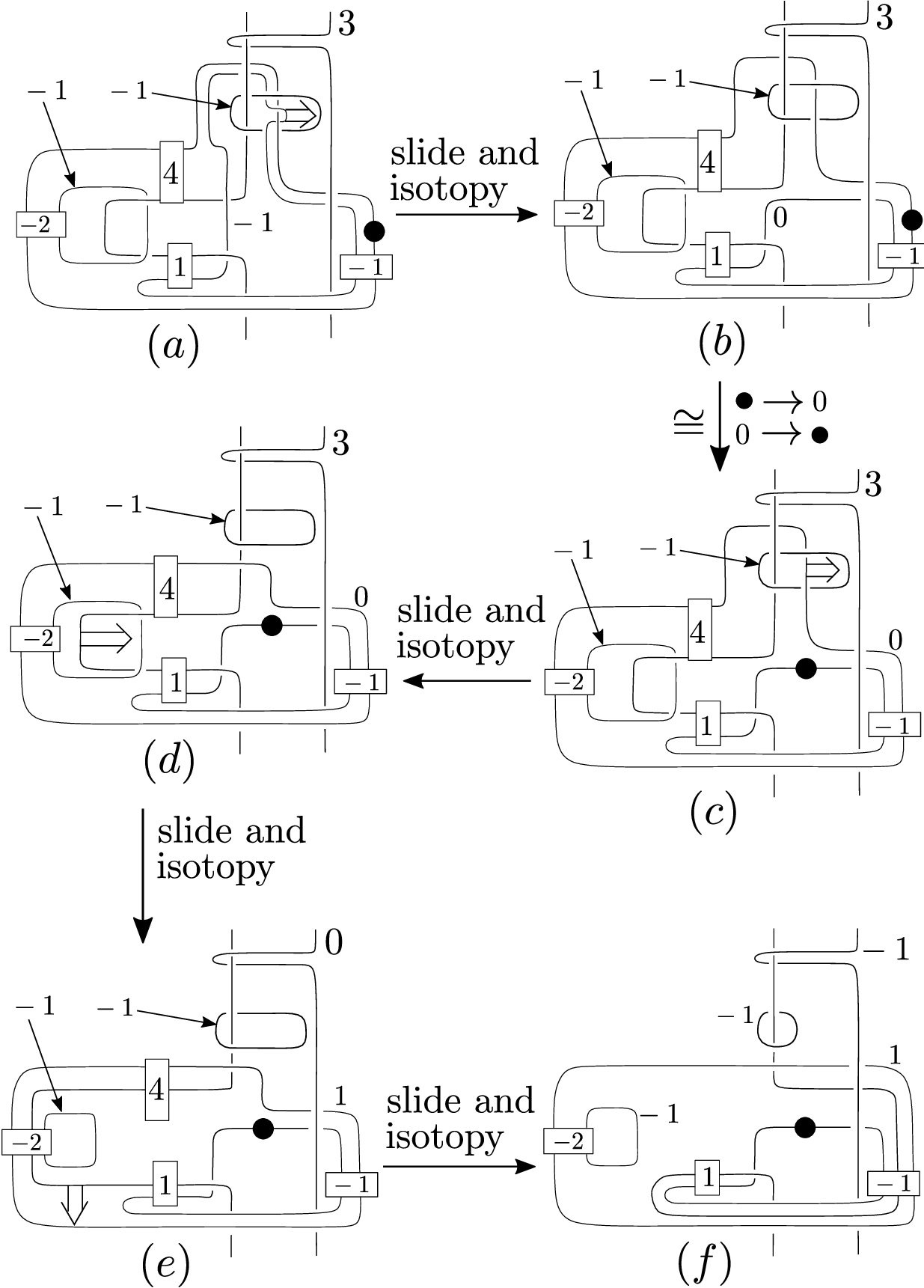}
        \caption{Local deformations of Figure \ref{fig:R8_corktwist}.}
        \label{fig:R8_corktwist_local_2}
      \end{center}
    \end{figure}
    \begin{figure}[htbp]
      \begin{center}
        \includegraphics[width=11cm]{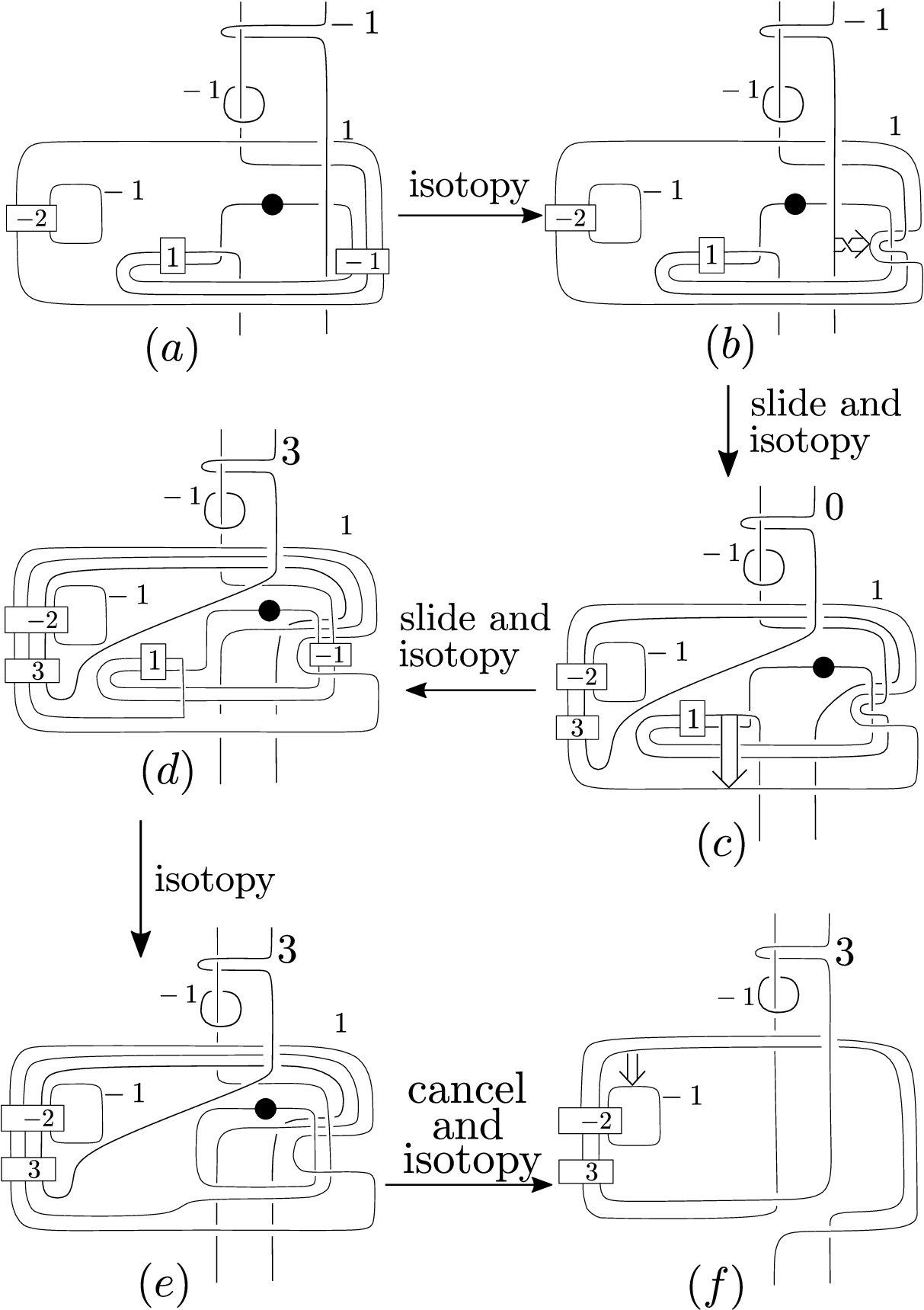}
        \caption{Local deformations of Figure \ref{fig:R8_corktwist}.}
        \label{fig:R8_corktwist_local_3}
      \end{center}
    \end{figure}
    \begin{figure}[htbp]
      \begin{center}
        \includegraphics[width=10cm]{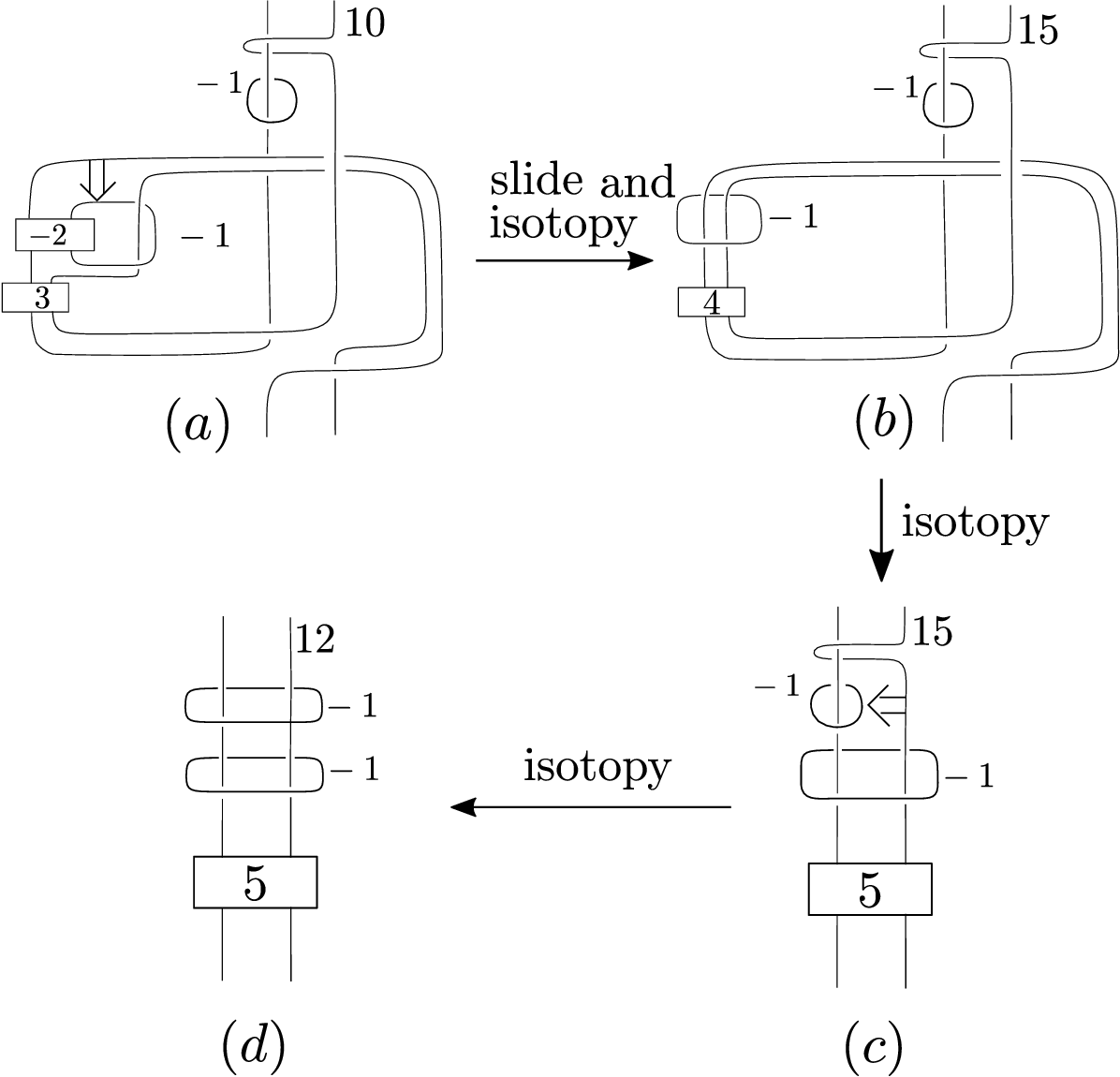}
        \caption{Local deformations of Figure \ref{fig:R8_corktwist}.}
        \label{fig:R8_corktwist_local_4}
      \end{center}
    \end{figure}
    \begin{figure}[htbp]
      \begin{center}
        \includegraphics[width=12cm]{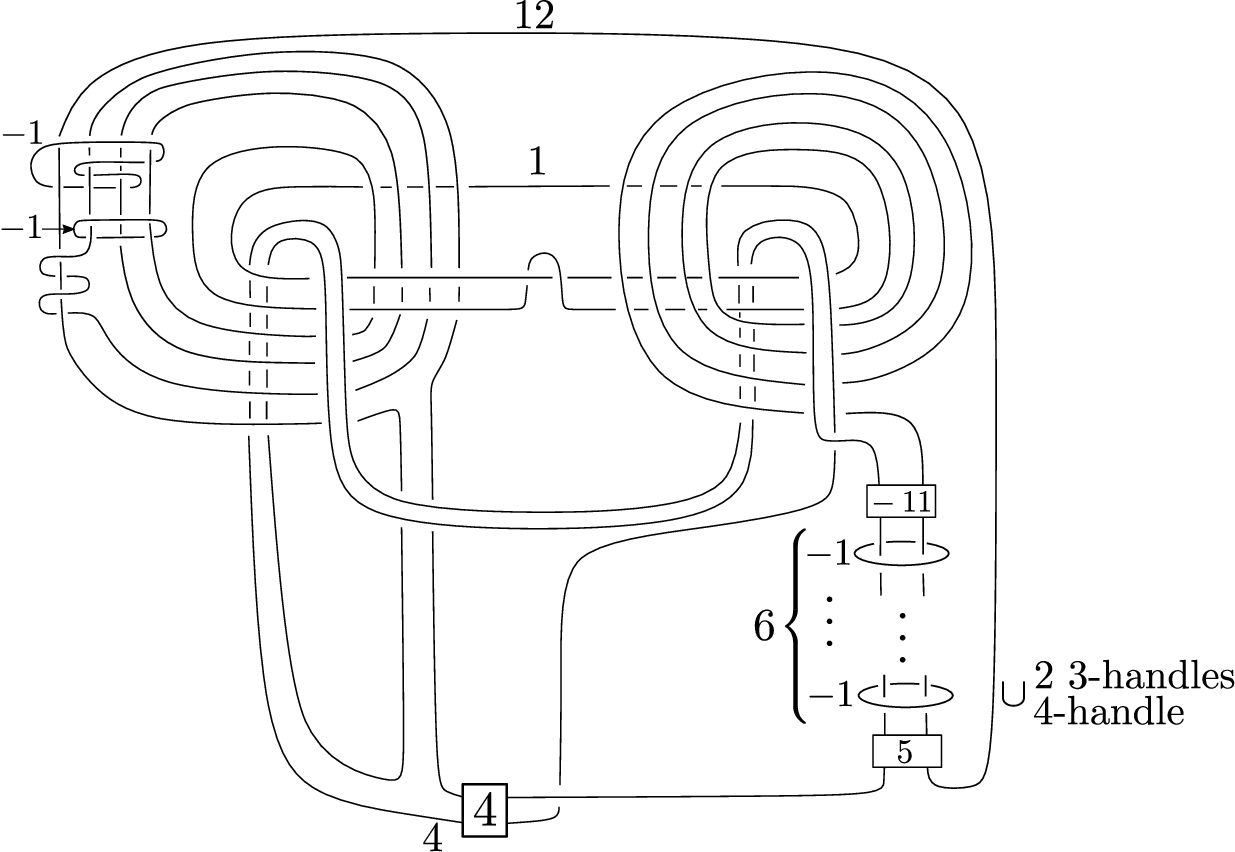}
        \caption{The entire picture of Figure \ref{fig:R8_corktwist_local_4} (d).}
        \label{fig:CP2-8CP2_1}
      \end{center}
    \end{figure}
    \begin{figure}[htbp]
      \begin{center}
        \includegraphics[width=12cm]{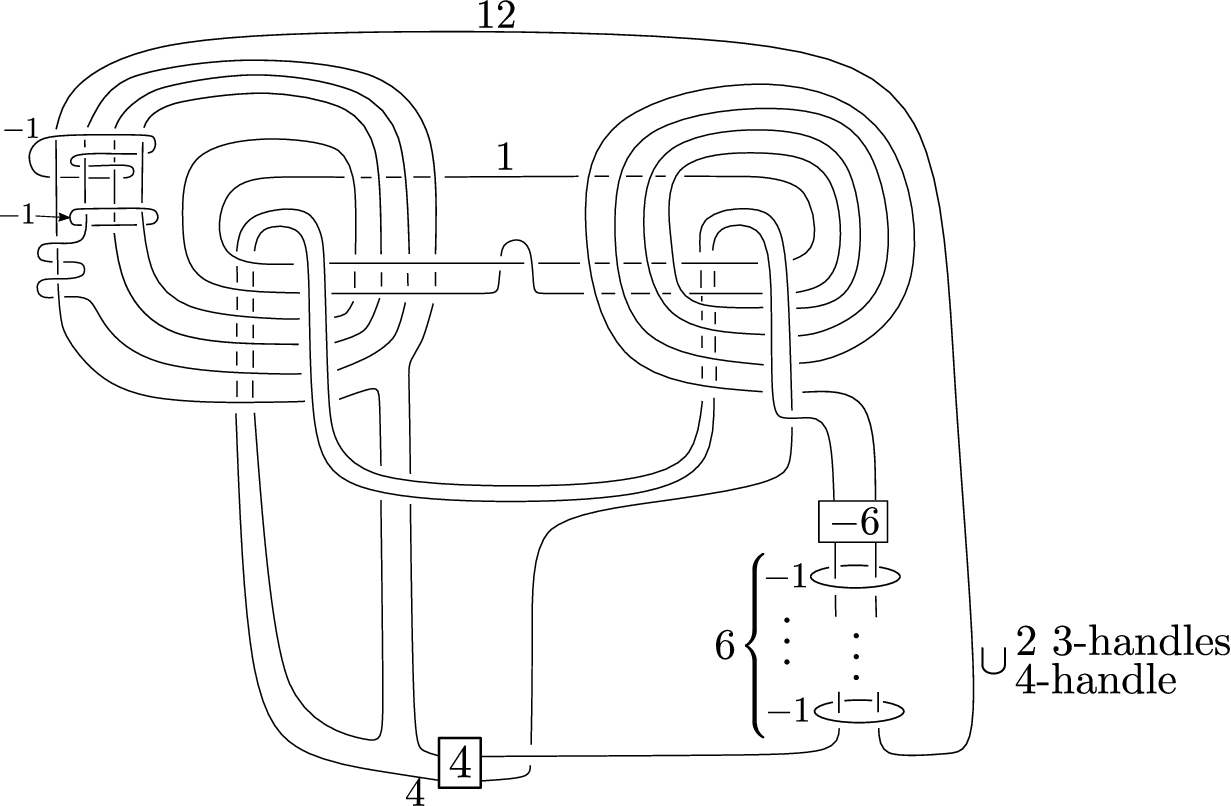}
        \caption{A diagram of $\mathbf{CP}^2 \# 8\overline{\mathbf{CP}^2}$.}
        \label{fig:CP2-8CP2_2}
      \end{center}
    \end{figure}

  \begin{proof}[Proof of Corollary \ref{cor:stab}]
    First, recall that if $M$ is a simply-connected $4$-manifold, the result of the surgery along any embedded $S^1$ must be diffeomorphic to either $M\# S^2\! \times \! S^2$ or $M\# S^2 \tilde{\times} S^2$ (\cite[Proposition 5.2.3]{1999Gompf-Stipsicz}). Furthermore, if $M$ is a non-spin simply-connected $4$-manifold, then $M\# S^2 \! \times \! S^2$ and $M\# S^2 \tilde{\times} S^2$ are diffeomorphic (\cite[Proposition 5.2.4]{1999Gompf-Stipsicz}). Since the $4$-manifolds $R_8$ and $\mathbf{CP}^2 \# 8\overline{\mathbf{CP}^2}$ are non-spin and simply-connected, we only have to show that the results of the surgeries along embedded $S^1$ in those $4$-manifolds are diffeomorphic.

    By Theorem \ref{thm:cork}, we know that there exist diagrams of $R_8$ and $\mathbf{CP}^2 \# 8\overline{\mathbf{CP}^2}$ such that each of diagrams contains a copy of the diagram of $W_2$ (in Figure \ref{fig:W2}), and becomes identical after changing a dotted circle into a $0$-framed unknot. In general, changing a dotted circle into a $0$-framed unknot corresponds to a surgery along an embedded $S^1$ \cite[Section 5.4]{1999Gompf-Stipsicz}. Therefore the results of the surgeries along embedded $S^1$ in $R_8$ and $\mathbf{CP}^2 \# 8\overline{\mathbf{CP}^2}$ are diffeomorphic.
  \end{proof}

\end{document}